\documentclass[onefignum,onetabnum]{siamart220329}

\headers{Gaussian RBF collocation for fractional PDEs}{Xiaochuan Tian, Yixuan Wu, Yanzhi Zhang}

\title{Gaussian Radial Basis Functions Collocation for Fractional PDEs: Methodology and Error Analysis\thanks{Submitted to the editors DATE.
\funding{This work was partially supported by 
 NSF grant DMS-2111608, DMS-2240180, DMS–1913293 and DMS–1953177. }}}

\author{Xiaochuan Tian\thanks{Department of Mathematics, University of California, San Diego, CA 92093, United States
(\email{xctian@ucsd.edu}).}
\and Yixuan Wu\thanks{Department of Pharmacology, University of California Davis, Davis, CA 95616,   United States (\email{ywxwu@ucdavis.edu}) } 
\and Yanzhi Zhang\thanks{Department of Mathematics and Statistics, Missouri University of Science and Technology, Rolla, MO 65409,  United States (\email{zhangyanz@mst.edu)}}}

\usepackage{multirow}
\usepackage{amsmath}
\usepackage{amssymb}
\usepackage{bm}

\usepackage{xcolor}
\usepackage{indentfirst}
\usepackage{graphicx}

\usepackage{mathrsfs,esint}
\usepackage{comment}
\usepackage{enumerate}
\usepackage{hyperref}

\usepackage{float}

\usepackage{subcaption}
\usepackage{pgfplots}
\pgfplotsset{compat=newest, compat/show suggested version=false}

\usepackage{tikz}
\usetikzlibrary{decorations.pathreplacing, decorations.pathmorphing, decorations.shapes}
\usetikzlibrary{arrows,calc}
\usetikzlibrary{intersections}
\usetikzlibrary{shapes.geometric}
\usepackage{etoolbox} 
\usepackage{listofitems}

\newcommand{\XT}[1]{{\color{orange} {\bf Xiaochuan:} [#1] }}
\newcommand{\YZ}[1]{{\color{purple} {\bf Yanzhi:} [#1] }}

\newcommand{\beq}{\begin{equation}}
\newcommand{\eeq}{\end{equation}}

\newcommand{\p}{\partial}
\newcommand{\Og}{\Omega}
\newcommand{\fl}[2]{\frac{#1}{#2}}
\newcommand{\nn}{\nonumber}
\newcommand{\Dt}{\Delta}
\newcommand{\bx}{{\bf x}}

\def\cE{\mathcal{E}}

\def\cI{\mathcal{I}}

\def\C{\mathbb{C}}

\def\R{\mathbb{R}}

\def\Z{\mathbb{Z}}

\def\eb{\bm e}

\def\jb{\bm j}
\def\kb{\bm k}

\def\mb{\bm m}
\def\nb{\bm n}

\def\xb{\bm x}
\def\yb{\bm y}
\def\zb{\bm z}

\newcommand{\al}{\alpha}

\newcommand{\ga}{\gamma}
\newcommand{\del}{\delta}

\newcommand{\vep}{\varepsilon}

\newcommand{\la}{\lambda}
\newcommand{\vphi}{\varphi}
\newcommand{\Om}{\Omega}
\newcommand{\ap}{\alpha}
\newcommand{\xib}{\bm\xi}
\newcommand{\mub}{\bm\mu}

\ifpdf
\hypersetup{
  pdftitle={Gaussian Radial Basis Functions Collocation for Fractional PDEs: Methodology and Error Analysis},
  pdfauthor={Xiaochuan Tian, Yixuan Wu, Yanzhi Zhang}
}
\fi

\begin{document}

\maketitle

\begin{abstract}
The paper introduces a new meshfree pseudospectral method based on Gaussian radial basis functions (RBFs) collocation to solve fractional Poisson equations. Hypergeometric functions are used to represent the fractional Laplacian of Gaussian RBFs, enabling an efficient computation of stiffness matrix entries. 
Unlike existing RBF-based methods, our approach ensures a Toeplitz structure in the stiffness matrix with equally spaced RBF centers, enabling efficient matrix-vector multiplications using fast Fourier transforms. 
We conduct a comprehensive study on the shape parameter selection, addressing challenges related to ill-conditioning and numerical stability. 
The main contribution of our work includes rigorous stability analysis and error estimates of the Gaussian RBF collocation method, representing a first attempt at the rigorous analysis of RBF-based methods for fractional PDEs to the best of our knowledge. 
We conduct numerical experiments to validate our analysis and provide practical insights for implementation.
\end{abstract}




\section{Introduction}
\label{section1} 

Let $\Om \subset {\mathbb R}^d$ be a $d$-dimensional open and bounded domain. 
We consider the following fractional Poisson equation with extended Dirichlet boundary conditions:
\beq\label{Poisson}
\begin{aligned}
(-\Delta)^{\frac{\al}{2}} u({\bf x}) = f({\bf x}), 
\qquad & \text{on} \  \Om \\
u(\bx) = g(\bx),\qquad & \text{on} \ \Og^c,
\end{aligned}
\eeq
where $\Om^c = {\mathbb R}^d\backslash\Om$ denotes the complement of $\Og$.  
The fractional Laplacian $(-\Dt)^\fl{\ap}{2}$ is defined by a hypersingular integral (also known as the integral fractional Laplacian) \cite{Landkof,Samko}: 
\beq\label{integralFL}
(-\Delta)^{\fl{\alpha}{2}}u(\bx) = C_{d,\alpha}\,{\rm P. V.}\int_{{\mathbb R}^d}\fl{u(\bx) - u({\bf y})}{|\bx -{\bf y}|^{d+\ap}}d{\bf y}, \quad \ \ \mbox{for} \ \  \ap \in (0, 2),
\eeq 
where ${\rm P. V.}$ stands for the principal value integral, and 
$$C_{d,\alpha}=\frac{2^{\ap-1} \ap\,\Gamma\big(\fl{\ap+d}{2}\big)}{\sqrt{\pi^{d}}\,\Gamma\big(1 -\fl{\ap}{2}\big)}$$
with $\Gamma(\cdot)$ being the Gamma function. 
The fractional Laplacian represents the infinitesimal generator of a symmetric $\ap$-stable L\'evy process. 
It can be also defined via a pseudo-differential operator with symbol $|\xib|^\ap$  \cite{Landkof,Samko}: 
\beq
\label{pseudo}
(-\Delta)^{\fl{\alpha}{2}}u({\bx}) = \mathcal{F}^{-1}\big[|\xib|^\alpha \mathcal{F}[u]\big], \qquad \mbox{for} \ \  \ap > 0, 
\eeq
where $\mathcal{F}$ is denotes the Fourier transform, and $\mathcal{F}^{-1}$ is its associated inverse transform. 
In a special case of $\ap = 2$, the operator in \eqref{pseudo} reduces to the spectral representation of the classical negative Laplacian $-\Dt$. 
For $\ap \in (0, 2)$, the two definitions of the fractional Laplacian in (\ref{pseudo}) and (\ref{integralFL}) are equivalent for functions in the Schwartz space ${\mathcal S}({\R}^d)$ \cite{Kwasnicki2017,Landkof,Samko}. 

So far, various numerical methods have been developed for the fractional Laplacian $(-\Dt)^\fl{\ap}{2}$, including finite element methods \cite{Delia2013, Tian2016, Acosta2017, Ainsworth2017, Bonito2019},  finite difference methods \cite{Huang2014, Duo2018, Duo2019-FDM, Minden2020, Hao2021, Wu2022}, and spectral methods  \cite{Kirkpatrick2016, Duo2016, Xu0018, Hao2021sharp}. 
Finite element and finite difference methods are typically formulated using the integral definition given in \cref{integralFL}.  
Consequently, a primary challenge lies in the accurate approximation of the hypersingular integrals. 
Furthermore, the entries of the stiffness matrix are usually determined by $d$- or $2d$-dimensional integrals, which makes the assembly of the stiffness matrix expensive, particularly when the dimension is large. 
On the other hand, spectral methods have gained increasing attention in solving nonlocal and fractional PDEs. 
They can achieve high accuracy with less number of points and thus effectively address the challenges caused by the nonlocality. 
Various spectral methods have been developed; see \cite{Duo2016, Hao2021sharp, Kirkpatrick2016, Xu0018, Sheng2020} and references therein. 
However, they are usually limited by simple geometry and boundary conditions. 
Recently, a new class of meshfree pseudospectral methods has been proposed based on Radial Basis Functions (RBFs) \cite{Burkardt2021, Wu0021, Wu2022, Zhuang0022}. 
These methods utilize the hypergeometric functions to represent the Laplacian of the RBFs and thus avoid numerical evaluation of hypersingular integrals associated with the fractional Laplacian.
Numerical experiments in these studies show high accuracy of the methods with a relatively small number of degree of freedom.
Despite their advantages, these methods face challenges such as ill-conditioned systems and unknown numerical stability. Furthermore, even when equal-spaced RBF center and test points are used, the stiffness matrices lose their Toeplitz structure.
The selection of shape parameters remains an open question, often addressed through ad-hoc choices in practical applications.

In this paper, we propose a new meshfree pseudospectral method based on the Gaussian RBFs. 
Similar to the previously developed RBF-based pseudospectral methods, we utilize hypergeometric functions to express the fractional Laplacian of Gaussian RBFs, thus avoiding the need for approximating the hypersingular integrals and maintaining a dimension-independent cost in computing stiffness matrix entries.  
Differing from the methods presented in \cite{Burkardt2021, Wu0021, Wu2022, Zhuang0022}, our approach yields a stiffness matrix with a Toeplitz structure when equally spaced RBF centers are employed. This enables the development of fast algorithms, such as those based on fast Fourier transform (FFT), for efficient matrix-vector multiplications.
Furthermore, we conduct a comprehensive study of the selection of the shape parameter, aiming to address the challenges related to ill-conditioning and numerical stability. 
Numerical studies confirm the effectiveness of our approach in mitigating the ill-conditioning issues associated with RBF-based methods.
Another significant contribution of our work is the rigorous error estimates for the proposed method. To the best of our knowledge, our work provides the first set of analytical results for RBF-based methods applied to solving fractional PDEs.
Our method involves coupling the shape parameter and the spatial discretization parameter.
Under appropriate coupling of the two parameters, we show that the consistency error of the numerical scheme comprises a part that converges to zero at a rate depending on the smoothness of the functions (in particular, the rate is faster than any polynomial rate for $C^\infty$ functions). Additionally,  there is a relatively small yet non-convergent part known as the saturation error. 
This aligns with the notion of ``approximate approximation'' as discussed in \cite{mazya2007approximate}. In addition, achieving stability for collocation methods for nonlocal equations poses a nontrivial challenge. 
Our analytical approach for numerical stability is inspired by \cite{LTTF21,LTTF20}, where the collocation scheme is compared with the Galerkin scheme to establish stability.
To leverage the stability of the Gaussian RBF-based Galerkin scheme, we establish a new fractional Poincar\'e type inequality for globally supported (Gaussian mixture) functions.  
Our study not only fills the fundamental gap of lack of analytical results for RBF-based methods in the literature but also provides a practice guide in the selection of the shape parameters. 

The paper is organized as follows. 
In \Cref{section-method}, we propose a new meshfree collocation method using Gaussian RBFs. 
We discuss the approximation of fractional Poisson equation with homogeneous Dirichlet boundary conditions. 
Non-homogeneous boundary value problems are converted into homogeneous ones through the introduction of an auxiliary function. 
\Cref{section-analysis} provides the truncation error and stability analyses for the proposed numerical method. 
In \Cref{section-experiments}, we test the performance of our method in solving fractional Poisson equations with homogeneous Dirichlet boundary conditions. \Cref{section-nonhomogeneous} is dedicated to the non-homogeneous boundary value problems, where we discuss the numerical practice of the approximation of the auxiliary function and present convergence studies of numerical solutions.
Conclusion and further discussions are presented in \Cref{section-conclusion}.

\section{Gaussian RBF collocation method}
\label{section-method}

In this section, we introduce a new meshfree collocation method based on Gaussian RBFs. 
We first note that \cref{Poisson} can be converted to a problem with homogeneous Dirichlet boundary conditions by subtracting the solution to the original equation from an auxiliary function that coincides with the boundary condition. More precisely, let $w : \R^d \to \R$ be a function of sufficient regularity such that $w|_{\Om^c} = g$.  
By letting $v = u -w$,  \cref{Poisson} is then equivalent to
\begin{equation}\label{Poisson-v}
\left\{ 
\begin{aligned}
(-\Delta)^{\frac{\al}{2}} v(\xb) = f(\xb) - (-\Delta)^{\frac{\al}{2}}w(\xb), \qquad   &\text{for} \ \ \xb \in \Om, \\
v(\xb) = 0, \qquad  &\text{for} \ \ \xb \in \Om^c .
\end{aligned}
\right.
\end{equation}
Therefore, we focus on the Poisson equation (\ref{Poisson}) with homogeneous Dirichlet boundary conditions (i.e., $g(\xb) \equiv 0$) in this section without loss of generality. 
Detailed discussions on the numerical approximation of the auxiliary function $w$ can be found in \Cref{section-nonhomogeneous}.

A Gaussian RBF is defined as \begin{equation}
\label{Gaussian}
\varphi^\varepsilon({\xb}) = \exp(-\varepsilon^2|\xb|^2), \qquad\text{for} \ \ \xb \in {\mathbb R}^d,
\end{equation}
where $\varepsilon \in {\mathbb R}$ represents the shape parameter. For the discussion on appropriate choices of $\vep$, we refer the readers to later sections. 
Let $N > 0$ be an integer and  $\{\xb_k\}_{k=1}^N \in \Og$ be a set of RBF centers. 
We assume that solution $u$ to \cref{Poisson} for $g\equiv 0$ can be approximated by
\begin{equation}\label{ansatz}
u_h(\xb) = \sum_{k=1}^{N}\lambda_k\, \vphi^{\vep}(\xb-\xb_k) = \sum_{k=1}^{N} \lambda_k\,e^{-\varepsilon^2|\xb - \xb_k|^2}, \qquad \mbox{for} \ \ \xb \in \Og. 
\end{equation}
Our RBF collocation method for approximating \cref{Poisson} with $g\equiv0$ is formulated as
\beq
\label{discrete-Poisson}
(-\Delta)^{\frac{\al}{2}} u_h (\xb'_j) = f(\xb'_j), \quad j=1,2,\cdots, N,
\eeq
where $\{ \xb'_j\}_{j=1}^N\subset \Og$
is a set of test points that may or may not coincide with the set of centers $\{ \xb_k\}_{k=1}^N$. 
Note that for the homogeneous Dirichlet boundary value problem, the RBF centers and test points are taken only from the domain $\Og$, and the discrete problem \eqref{discrete-Poisson} is not accompanied by any boundary conditions, differing from the strategy in \cite{Burkardt2021} where RBF centers and test points are located on $\overline{\Og}$.
A basic rationale for our strategy is that when
the shape parameter $\varepsilon$ is large,  $u_h(\xb) \approx 0$ for $\xb \in \Og^c$, that is, $u_h$ approximately satisfies the homogeneous boundary conditions. 
Rigorous analysis for the convergence of the numerical solution $u_h$ to the exact solution $u$ for appropriately chosen $\vep$ is presented in \Cref{section-analysis}.

A significant benefit of Gaussian RBFs is that the action of fractional Laplacian on them are expressed analytically by hypergeometric functions (\cite{Burkardt2021, Pearson2009, Pearson2016}): 
\begin{equation}\label{analy}
(-\Delta)^{\frac{\ap}{2}} e^{-\varepsilon^2|\xb|^2} =  \fl{2^\ap\Gamma\big(\fl{d+\ap}{2}\big)}{\Gamma\big(\fl{d}{2}\big)}|\varepsilon|^\ap\,_1F_1\Big(\fl{d+\ap}{2}; \, \fl{d}{2}; \, -\varepsilon^2|\xb|^2\Big), 
\end{equation}
for $\ap\geq 0$ and any dimension $d \ge 1$, where $_1F_1$ denotes the confluent hypergeometric function. 
Combining \eqref{ansatz} and \eqref{analy}, we immediately obtain
\begin{eqnarray}\label{FL-analytica}
(-\Delta)^\fl{\ap}{2}u_h(\xb) = \fl{2^\ap\Gamma\big(\fl{d+\ap}{2}\big)}{\Gamma\big(\fl{d}{2}\big)}|\varepsilon|^\ap\sum_{k = 1}^N \lambda_k \,_1F_1\Big(\fl{d+\ap}{2}; \, \fl{d}{2}; \, -\varepsilon^2|\xb-\xb_k|^2\Big),
\end{eqnarray}
for $\xb \in \Og$.
Denote ${\boldsymbol{\lambda}} = (\lambda_1, \lambda_2, \cdots, \lambda_N)^T$ and ${\bf f} = \big(f(\xb_1'), f(\xb_2'), \cdots, f(\xb_N')\big)^T$. 
The discrete problem (\ref{discrete-Poisson}) can be reformulated as a linear system $A \boldsymbol{\lambda} = {\bf f}$ with the entries of matrix $A=(a_{jk})_{j,k=1}^N$ given by
$$a_{jk} = \fl{2^\ap\Gamma\big(\fl{d+\ap}{2}\big)}{\Gamma\big(\fl{d}{2}\big)}|\varepsilon|^\ap\,_1F_1\Big(\fl{d+\ap}{2}; \, \fl{d}{2}; \, -\varepsilon^2|\xb_j'-\xb_k|^2\Big), \qquad \mbox{for} \ \ 1 \le j, k \le N.$$
Solving the linear system for ${\boldsymbol{\lambda}} = (\lambda_1, \lambda_2, \cdots, \lambda_N)^T$ and substituting them into (\ref{ansatz}), we can obtain the approximate solution to \cref{Poisson} with $g\equiv0$.

The above discussion shows that the dimension $d\geq1$ serves as a  parameter of the confluent hypergeometric function $_1F_1$. 
Hence, the computational cost in calculating entry $a_{kj}$ is independent of the dimension. 
This is one of the main advantages of our method compared to many other methods found in the literature, where the entries of the stiffness are determined by $d$- or $2d$-dimensional integrals \cite{Acosta2017, Ainsworth2018,DDGG20,  Duo2018, Hao2021, Huang2014, Minden2020}. 
Consequently, as the dimension $d$ increases, the computational cost associated with matrix assembly in these methods increase rapidly. 
Note that if test points $\xb'_j$ and center points $\xb_k$ are chosen from the same set of uniform grid points, the matrix $A$ is a (multi-level) symmetric Toeplitz matrix. In this case, fast algorithms for matrix-vector multiplication through FFT can be utilized to further reduce the computational cost. In this paper, we consider only the case where the test points coincide with the center points and are selected from uniform grid points.
This assumption not only facilitates fast algorithms but also allows for convergence analysis (detailed in \Cref{section-analysis}) of the RBF collocation method. Exploring more general cases remains a topic for future investigation.

Finally, we remark that our new method differs from the approach presented in \cite{Burkardt2021} in several key aspects. 
First, a crucial distinction lies in how boundary conditions are handled in the two approaches.
 In our method, homogeneous Dirichlet boundary conditions require no specific treatment, given that the solution ansatz \cref{ansatz} approximately satisfies the boundary conditions $g(\xb) = 0$ for $\xb\in\Og^c$ with appropriately large $\vep$.
 The treatment of non-homogeneous boundary conditions involves transforming them into homogeneous cases through the use of an auxiliary function.
 This particular boundary treatment allows for a Topelitz structure of the stiffness matrix when we use uniform grid points,  a crucial property that enables fast algorithms.  
Additionally, we employ a systematic approach to select the shape parameter $\vep$; further details will be provided later.
With appropriate coupling of the shape parameter $\vep$ with the mesh parameter $h$, it will be shown in \Cref{section-experiments} that the resulting linear systems exhibit significantly improved conditioning compared to those in \cite{Burkardt2021}.
The boundary treatment and the coupling of the shape parameter and the discretization parameter also allow for the numerical analysis of this approach, which will be detailed in \Cref{section-analysis}. To the best of our knowledge, this presents a first work for the analysis of RBF collocation method applied to fractional PDEs.

\section{Convergence analysis}
\label{section-analysis}
In this section, we discuss convergence analysis of the Gaussian RBF collocation method proposed in \Cref{section-method}.
Based on the discussion in \Cref{section-method}, we focus on the fractional Poisson equation with homogeneous Dirichlet boundary condition. The numerical construction of the auxiliary function $w$ for non-homogeneous problems will be addressed in subsequent sections. 
For a collection of collocation points $\{ \xb_i\}_{i=1}^N\subset \Om$, we define the associated finite dimensional space 
\beq
\label{eq:FiniteSoluSpace}
\begin{split}
V_h : = \left\{ v= \sum_{i=1}^N \lambda_i \vphi^\varepsilon(\xb-\xb_i): \la_i \in \R, i=1,2,\cdots,N.\right\} .  
\end{split}
\eeq 
We will analyze the convergence of the Gaussian RBF collocation method concerning the points $\{ \xb_i\}_{i=1}^N$ on a uniform grid, i.e., $\xb_i = \kb h$ for some $\kb\in\Z^d$ and grid size $h>0$. 
The numerical analysis for meshfree points in a more general context will be explored in future studies. Our approach hinges on the crucial idea of maintaining $\vep h$ as a constant, facilitating both the analysis of truncation error and stability. In the subsequent parts of this section, we conduct an examination of truncation error (in \Cref{subsection-truncation}) and stability (in \Cref{subsection-stability}) of the numerical scheme. Finally, we present the convergence result in \Cref{subsection-convergence}.
\subsection{Truncation error analysis}
\label{subsection-truncation}
In this subsection, we perform truncation error analysis.
For any continuous function $u:\R^d\to \R$, we define $\cI_h u$ as the RBF interpolation of $u$ 
on a uniform lattice of grid size $h$, i.e., 
\begin{equation}
\label{eq:interpolation_def}
\cI_h u(\xb) = \sum_{\mb\in\Z^d} a_{\mb} \vphi^\vep(\xb - h\mb)
\end{equation}
such that \, $\cI_h u (h\mb) = u(h\mb)$ for all $\mb\in \Z^d$. 
An important insight from the existing literature highlights that such interpolating functions, while maintaining $\vep h$ as a constant, are accurate without being convergent in a rigorous sense. More specifically, the interpolation error consists of a part converging to zero as $h\to 0$, and a relatively small yet non-convergent part referred to as the {\it saturation error}. Such approximation procedures are termed ``approximate approximations'' in the work by Maz'ya and Schmidt \cite{mazya2007approximate}. 
Our goal in this subsection is to extend the error analysis in \cite{mazya2007approximate} for the approximation of function to study the truncation error $|(-\Delta)^{\frac{\al}{2}} u  - (-\Delta)^{\frac{\al}{2}} \cI_h u|$. In the end, it is shown in \Cref{thm:consistency} that the truncation error also consists of two parts -- a part converging to zero and another error part of saturation. 

We first recall some of the relevant results in \cite{mazya2007approximate}.
Throughout the rest of this paper, we denote
$$\gamma = (\vep h)^2,$$ 
and assume $\ga>0$ is a fixed constant. 
Following \cite[Chapter 7.3]{mazya2007approximate}, one can first rewrite $\cI_h u$ by Lagrangian functions, which are accurately approximated by 
\beq
\Psi_\gamma(\xb): = \frac{\gamma^d}{\pi^d} \prod_{j=1}^d \frac{\sin (\pi x_j)}{\sinh (\gamma x_j)}
\eeq
with $\xb = (x_1, x_2, \cdots, x_d)$. The Fourier transform of $\Psi_\gamma$ is given by
\beq
\label{eq:PsiTransform}
\widehat{\Psi}_\gamma(\xib) = \prod_{j=1}^d \frac{\sinh(\pi^2 \gamma^{-1})}{\cosh(\pi \gamma^{-1} \xi_j) + \cosh(\pi^2 \gamma^{-1}) }.
\eeq
Then instead of \cref{eq:interpolation_def}, we can represent $\cI_h u$ by (comment on error?)
\beq
\label{eq:interpolation}
\cI_h u (\xb) = \sum_{\mb\in \Z^d} u(h\mb) \Psi_\gamma(\frac{\xb}{h} - \mb).
\eeq

In the following, we adopt the standard multi-index notations for the presentation in this section. More precisely, a multi-index is a vector of non-negative integers, i.e., ${\bm\ap} = (\ap_1, \ap_2, \cdots, \ap_d)$ with $\ap_k \in {\mathbb N}^0$. Its length is defined as $|{\bm \al}| := \sum_{k=1}^d \al_k$, and we denote 
\beq{\bm\al}! = \prod_{k=1}^d \ap_k!, \qquad \xb^{\bm\ap} = \prod_{k=1}^d x_k^{\ap_k}, \qquad
 D^{\bm\ap}u(\xb) = \frac{\partial^{|{\bm \ap}|}}{\partial_{x_1}^{\ap_1} \partial_{x_2}^{\ap_2} \cdots  \partial_{x_d}^{\ap_d}}u(\xb).\nonumber
\eeq
Extending the result in \cite[Chapter 7.3]{mazya2007approximate}, we present a derivative estimate for function approximation by $\cI_h u$ defined in \cref{eq:interpolation}.
\begin{lemma}[approximation errors]
\label{lem:approx_err}
Suppose $M$ is a positive integer and
\begin{equation*}
|u|_M : = \int_{\R^d} |\widehat{u}(\xib)| |\xib|^M d\xib  < \infty. 
\end{equation*}
For a multi-index $\bm\beta$ with $0\leq |\bm\beta|\leq M$, we have the following estimate
\begin{equation*}
|D^{\bm\beta} u(\xb) - D^{\bm\beta} \cI_h u(\xb)| \leq c\,| u |_{M} h^{M-|\bm\beta|} + \sum_{|\bm\al|=0}^{M-1}\frac{|a^{(\bm\beta)}_{\bm\al}(\xb/h)|}{\bm\al !} |h|^{|\bm\al|-|\bm\beta|} |D^{\bm\al} u(\xb)|,  
\end{equation*}
where 
\begin{equation*}
a^{(\bm\beta)}_{\bm\al}(\xb) = D^{\bm\al}_{\xib} \big(\xib^{\bm\beta} - \theta_{\gamma}^{(\bm\beta)}(\xb, \xib)\big)|_{\xib = \bm{0},}
\end{equation*}
with
\begin{equation*}
\theta_{\gamma}^{(\bm\beta)}(\xb, \xib) = \sum_{\nb\in\Z^d} (\xib+ 2\pi\nb)^{\bm\beta} \widehat{\Psi}_\gamma(\xib + 2\pi\nb) e^{2\pi i \xb \cdot \nb}.
\end{equation*}
\end{lemma}
\begin{proof}
For any multi-index $\bm\beta$, by inverse Fourier transform we have
\begin{equation*}
D^{\bm\beta} u(\xb) = \frac{1}{(2\pi)^d} D^{\bm\beta}\bigg(\int_{\R^d} \widehat{u}(\xib) e^{i \xb\cdot\xib} d\xib\bigg) = \frac{i^{|\bm\beta|}}{(2\pi)^d}  \int_{\R^d} \widehat{u}(\xib) \xib^{\bm\beta} e^{i \xb\cdot\xib} d\xib,
\end{equation*}
and
\begin{equation*}
\begin{split}
D^{\bm\beta} \cI_h u(\xb) &= \sum_{\mb\in\Z^d} u(h\mb) h^{-|\bm\beta|}  D^{\bm\beta}\Psi_\gamma \big(\frac{\xb}{h} - \mb\big) \\
&= \frac{1}{(2\pi)^d}  \sum_{\mb\in\Z^d} \left(\int_{\R^d} \widehat{u}(\xib) e^{i h\mb\cdot \xib} d\xib\right) h^{-|\bm\beta|} D^{\bm\beta}\Psi_\gamma\big(\frac{\xb}{h} - \mb\big) \\
&=\frac{1}{(2\pi)^d} \int_{\R^d} \widehat{u}(\xib) \bigg(\sum_{\mb\in\Z^d} h^{-|\bm\beta|} D^{\bm\beta}\Psi_\gamma\big(\frac{\xb}{h} - \mb\big)\,e^{ i h\mb\cdot \xib} \bigg) d\xib .
\end{split}
\end{equation*}
Recall the Poisson summation formula (see e.g., \cite{pinsky2008introduction}) 
\begin{equation*}
\sum_{\mb\in\Z^d} f(\mb) =\sum_{\nb\in\Z^d}  \widehat{f}(2\pi\nb).
\end{equation*}
We have
\begin{equation*}
D^{\bm\beta} \cI_h u(\xb) = \frac{i^{|\bm\beta|}}{(2\pi)^d}
 \int_{\R^d} \widehat{u}(\xib) e^{ i \xb\cdot \xib} \sum_{\nb\in\Z^d} h^{-|\bm\beta|} e^{ i 2\pi\frac{\xb}{h} \cdot \nb}   (h\xib + 2\pi\nb)^{\bm\beta} \widehat{\Psi}_\gamma(h\xib + 2\pi\nb).
\end{equation*}
By the definition of $\theta_{\gamma}^{(\bm\beta)}$, we then obtain
\begin{eqnarray}
&&D^{\bm\beta} u(\xb) - D^{\bm\beta} \cI_h u(\xb)  = \frac{i^{|\bm\beta|}}{(2\pi)^d}\int_{\R^d} \widehat{u}(\xib) h^{-|\bm\beta|} \left[ (h\xib)^{\bm\beta} -\theta_{\gamma}^{(\bm\beta)}\big(\frac{\xb}{h}, h\xib\big) \right] e^{i \xb\cdot\xib}\,d\xib\nonumber\\
&&=\scalebox{0.95}{$\frac{i^{|\bm\beta|}}{(2\pi)^d}\sum_{\mub\in\Z^d}\int_{h^{-1}Q}  \widehat{u}(\xib+ \frac{2\pi\mub}{h}) h^{-|\bm\beta|}\left[ (h\xib+ 2\pi\mub)^{\bm\beta} -\theta_{\gamma}^{(\bm\beta)}\big(\frac{\xb}{h}, h\xib+2\pi\mub\big) \right] e^{i \xb\cdot(\xib+\frac{2\pi\mub}{h})}   d\xib$} \nonumber \\
&& = \text{I} + \text{II}, \nonumber 
\end{eqnarray}
where we denote $Q = (-\pi, \pi)^d$, and
\begin{equation*}
\begin{split}
&\text{I} = \frac{i^{|\bm\beta|}}{(2\pi)^d} \int_{h^{-1}Q}\widehat{u}(\xib) h^{-|\bm\beta|}  \Big[ (h\xib)^{\bm\beta} -\theta_{\gamma}^{(\bm\beta)}\big(\frac{\xb}{h}, h\xib\big) \Big] e^{ i \xb\cdot\xib}\,d\xib , \\
&\text{II} = \\
&
\scalebox{0.92}{$\frac{i^{|\bm\beta|}}{(2\pi)^d}  \sum_{\mub \in\Z^d\backslash\{\bm{0}\}}\int_{h^{-1}Q}\widehat{u}(\xib+ \frac{2\pi\mub}{h}) h^{-|\bm\beta|}\left[ (h\xib+ 2\pi\mub)^{\bm\beta} -\theta_{\gamma}^{(\bm\beta)}\big(\frac{\xb}{h}, h\xib+2\pi\mub\big) \right] e^{i \xb\cdot(\xib+\frac{2\pi\mub}{h})} \, d\xib$}. 
\end{split}
\end{equation*}

Now we estimate term I following the proof of \cite[Theorem 7.10]{mazya2007approximate}.
Notice that by \cref{eq:PsiTransform}, one can write 
\begin{equation*}
\theta_{\gamma}^{(\bm\beta)}(\xb, \xib) = \sum_{\nb\in\Z^d} \prod_{j=1}^d \frac{(\xi_j + 2\pi n_j)^{\beta_j}\sinh(\pi^2 \gamma^{-1}) e^{2\pi i x_j n_j}}{\cosh(\pi \gamma^{-1} (\xi_j+2\pi n_j)) + \cosh(\pi^2 \gamma^{-1}) }.
\end{equation*}
Therefore, for a fixed $\xb\in\R^d$, the function $\theta_{\gamma}^{(\bm\beta)}(\xb, \zb)$ is meromorphic in $\zb\in\C^d$ with simple poles at
\begin{equation*}
2\pi \kb + \pi \eb + i 2\ga (\mb + \frac{1}{2} \eb) \quad \kb, \mb \in \Z^d
\end{equation*}
where $\eb =(1,1,\cdots, 1)\in \R^d$.  
Noticing the definition of $a^{(\bm\beta)}_{\bm\al}(\xb)$, the series 
\beq
\label{eq:ThetaSeries}
\zb^{\bm\beta} - \theta_{\gamma}^{(\bm\beta)}(\xb, \zb) = \sum_{|\bm\al|=0}^\infty a^{(\bm\beta)}_{\bm\al}(\xb)\frac{\zb^{\bm\al}}{\bm\al !},
\eeq
is absolutely convergent for $\zb = (z_1,z_2,\cdots, z_d)\in \C^d$ with 
\begin{equation*}
|z_i|\leq \sqrt{\pi^2+\ga^2} -\delta
\end{equation*}
for any small fixed 
 constant $\delta>0$. 
Then
\begin{equation*}
\begin{split}
|\text{I}| &= \frac{1}{(2\pi)^d} \bigg| \sum_{|\bm\al|=0}^\infty \frac{a^{(\bm\beta)}_{\bm\al}(\frac{\xb}{h})}{\bm\al !} h^{|\bm\al|-|\bm\beta|}\int_{h^{-1}Q}\widehat{u}(\xib) \xib^{\bm\al} e^{i \xb\cdot\xib}  d\xib \bigg| \\
&= \bigg|\frac{1}{(2\pi)^d}\sum_{|\bm\al|=0}^{M-1} \frac{a^{(\bm\beta)}_{\bm\al}(\frac{\xb}{h})}{\bm\al !} h^{|\bm\al|-|\bm\beta|}\bigg(\int_{\R^d}\widehat{u}(\xib) \xib^{\bm\al} e^{i \xb\cdot\xib}  d\xib - \int_{\R^d\backslash h^{-1}Q}\widehat{u}(\xib) \xib^{\bm\al} e^{i \xb\cdot\xib}  d\xib\bigg)\bigg| \\
&\qquad \qquad +\bigg|\frac{1}{(2\pi)^d}\sum_{|\bm\al|=M}^{\infty} \frac{a^{(\bm\beta)}_{\bm\al}(\frac{\xb}{h})}{\bm\al !} h^{|\bm\al|-|\bm\beta|}\int_{h^{-1}Q}\widehat{u}(\xib) \xib^{\bm\al} e^{i \xb\cdot\xib}  d\xib \bigg| \\
& \leq \sum_{|\bm\al|=0}^{M-1}\frac{|a^{(\bm\beta)|}_{\bm\al}(\xb/h)|}{\bm\al !} |h|^{|\bm\al|-|\bm\beta|} |D^{\bm\al} u(\xb)| + c_1(M) h^{M-|\bm\beta|} | u|_{M}. 
\end{split}
\end{equation*}
where $c_1(M) = \max_{\xb} \sum_{|\bm\al|=0}^{\infty}\frac{|a^{(\bm\beta)|}_{\bm\al}(\xb/h)|}{\bm\al !} \pi^{|\bm{\al}|-M}$ is well-defined because of \cref{eq:ThetaSeries}.  
The last inequality in the above is true since for $|\bm\al|<M$,
\begin{equation*}
\left| \int_{\R^d\backslash h^{-1}Q}\widehat{u}(\xib) \xib^{\bm\al} e^{i \xb\cdot\xib}  d\xib \right| \leq  \pi^{|\bm\al|-M} h^{M-|\bm\al|} \int_{\R^d\backslash h^{-1}Q} |\widehat{u}(\xib)|  |\xib|^{M}  d\xib,
\end{equation*}
and for $|\bm\al|\geq M$,
\begin{equation*}
\left|\int_{h^{-1}Q}\widehat{u}(\xib) \xib^{\bm\al} e^{i \xb\cdot\xib}  d\xib\right| \leq \pi^{|\bm\al|-M} h^{M-|\bm\al|} \int_{h^{-1}Q} |\widehat{u}(\xib)|  |\xib|^{M}  d\xib
\end{equation*}
Now we estimate term II.  
Notice that $|{\bm\beta}| \le M$ and
\begin{equation*}
\begin{split}
&\bigg| \frac{i^{|\bm\beta|}}{(2\pi)^d}  \sum_{\mub \in\Z^d\backslash\{\bm{0}\}}\int_{h^{-1}Q}\widehat{u}(\xib+ \frac{2\pi\mub}{h}) h^{-|\bm\beta|}(h\xib+ 2\pi\mub)^{\bm\beta} e^{i \xb\cdot(\xib+\frac{2\pi\mub}{h})} d\xib  \bigg| \\
\leq& \,
\frac{1}{(2\pi)^{d}}\int_{\R^d\backslash h^{-1}Q}  |\widehat{u}(\xib)| |\xib|^{|\bm\beta|} d\xib \leq 2^{-d}\pi^{|\bm\beta|-M-d} h^{M-|\bm\beta|} |u|_{M}. 
\end{split}
\end{equation*}
Therefore, we only need to estimate the term that involves $\theta_{\gamma}^{(\bm\beta)}(\frac{\xb}{h}, h\xib+2\pi\mub)$ in II. Notice that 
\begin{equation*}
\begin{split}
\theta_{\gamma}^{(\bm\beta)}(\frac{\xb}{h}, h\xib+2\pi\mub) = &\sum_{\nb\in\Z^d}  \big(h\xib+2\pi(\mub+ \nb)\big)^{\bm\beta}\,\widehat{\Psi}_\gamma\big(h\xib +2\pi(\mub+ \nb)\big)\,e^{2\pi i \frac{\xb}{h} \cdot \nb}\\
=& \sum_{\nb\in\Z^d} \big(h\xib+ 2\pi\nb\big)^{\bm\beta}\,\widehat{\Psi}_\gamma\big(h\xib + 2\pi\nb\big)\,e^{2\pi i \frac{\xb}{h} \cdot (\nb-\mub)}.
\end{split}
\end{equation*}
We therefore have
\begin{equation*}
\begin{split}
& \bigg| \frac{i^{|\bm\beta|}}{(2\pi)^d}  \sum_{\mub \in\Z^d\backslash\{\bm{0}\}}\int_{h^{-1}Q}\widehat{u}(\xib+ \frac{2\pi\mub}{h}) h^{-|\bm\beta|}\,\theta_{\gamma}^{(\bm\beta)}\big(\frac{\xb}{h}, h\xib+2\pi\mub\big)\,e^{2\pi i \xb\cdot(\xib+\frac{2\pi\mub}{h})}   d\xib \bigg| \\
\leq & \  \fl{1}{(2\pi)^{d}}\sum_{\mub \in\Z^d\backslash\{\bm{0}\}}\int_{h^{-1}Q} \big|\widehat{u}(\xib+ \frac{2\pi\mub}{h})\big| h^{-|\bm\beta|} \sum_{\nb\in\Z^d}|h\xib+ 2\pi\nb|^{|\bm\beta|} \big|\widehat{\Psi}_\gamma(h\xib + 2\pi\nb)\big| d\xib . 
\end{split}
\end{equation*}
We claim that there exists a constant $C>0$ such that $|\widehat{\Psi}_\gamma(h\xib + 2\pi\nb)|\leq C |\widehat{\Psi}_\gamma(2\pi\nb)|$ for $\xib\in h^{-1}Q$ by the definition in \eqref{eq:PsiTransform}. Indeed, this is true since $\widehat{\Psi}_\gamma(\xib)$ is a Lipschitz function for $\xib\in\R^d$ with a Lipschitz constant $L>0$, and in addition $|\widehat{\Psi}_\gamma(\xib)|\geq c$ for some $c>0$ for all $\xib\in\R^d$. The claim is then implied by
\begin{equation*}
|\widehat{\Psi}_\gamma(h\xib + 2\pi\nb) - \widehat{\Psi}_\gamma(2\pi\nb)| \leq L |h\xib| \leq \pi L \leq \pi L c^{-1} |\widehat{\Psi}_\gamma(2\pi\nb)|. 
\end{equation*}
In addition, for $\xib \in h^{-1}Q$, we get $$|h\xib+ 2\pi\nb|^{|\bm\beta|} \leq 2^{|\bm\beta|-1} \pi^{|\bm\beta|}(1 + |2\nb|^{|\bm\beta|}) = (2\pi)^{|\bm\beta|} 2^{-1}(1+ |2\nb|^{|\bm\beta|}).$$  Therefore
\begin{equation*}
\begin{split}
& \fl{1}{(2\pi)^{d}} \sum_{\mub \in\Z^d\backslash\{\bm{0}\}}\int_{h^{-1}Q} \big|\widehat{u}(\xib+ \frac{2\pi\mub}{h})\big| h^{-|\bm\beta|} \sum_{\nb\in\Z^d}|h\xib+ 2\pi\nb|^{|\bm\beta|} \big|\widehat{\Psi}_\gamma(h\xib + 2\pi\nb)\big| d\xib \\
\leq & C(2\pi)^{|\bm\beta|-d}  h^{-|\bm\beta|}\bigg(\sum_{\nb\in\Z^d} \big(1 + |2\nb|^{|\bm\beta|}\big) \widehat{\Psi}_\gamma(2\pi\nb)\bigg) \int_{\R^d\backslash h^{-1}Q} |\widehat{u}(\xib)| d\xib \leq  C h^{M-|\bm\beta|} |u|_{M}.
\end{split}
\end{equation*}
\end{proof}


We now present the truncation error analysis below. 

\begin{theorem}[consistency error]
\label{thm:consistency}
Suppose $M \in \Z^+$ and $|u|_M<\infty$.Then
\begin{equation*}
\begin{split}
\| (-\Delta)^{\frac{\al}{2}} u  - (-\Delta)^{\frac{\al}{2}} \cI_h u\|_{L^\infty(\R^d)}& \leq  
C h^{M-2} |u|_{M}  \\
&+ \sum_{|\bm\beta|\in\{ 0,2\}}\sup_{\xb\in\R^d}\sum_{|\bm\al|=0}^{M-1}\frac{|a^{(\bm\beta)}_{\bm\al}(\frac{\xb}{h})|}{\bm\al !} h^{|\bm\al|-|\bm\beta|} |D^{\bm\al} u(\xb)|.
\end{split}
\end{equation*}
    
\end{theorem}
\begin{proof}
By definition \eqref{integralFL}, we obtain
\begin{equation*}
 \begin{split}
 &\big| (-\Delta)^{\frac{\al}{2}} u(\xb) -  (-\Delta)^{\frac{\al}{2}} \cI_h u(\xb) \big|  \\
 \leq &C_{d,\alpha}\left(  \int_{B_r(\xb)} +  \int_{\R^d\backslash B_r(\xb)} \right) \frac{(u-\cI_hu)(\xb) - (u-\cI_hu)(\yb) }{|\yb -\xb|^{d+\al}} d\yb  \\
 \leq &C \bigg( r^{2-\alpha} \sum_{|\bm\beta| =2}\|D^{\bm\beta} u - D^{\bm\beta}\cI_h u \|_{L^\infty(\R^d)} + r^{-\alpha} \| u - \cI_h u \|_{L^\infty(\R^d)} \bigg), 
 \end{split}
\end{equation*}
 where $r>0$. Since $r$ can be any positive number, we choose $r=1$ for simplicity.
 Then the consistency error is obtained by applying \Cref{lem:approx_err}.
\end{proof}

\subsection{Stability analysis}
\label{subsection-stability}
In this subsection, we show the numerical stability of the Gaussian RBF collocation method.
Achieving stability for collocation methods for nonlocal equations poses a nontrivial challenge, primarily owing to the absence of a discrete maximum principle. A useful perspective involves comparing collocation schemes with Galerkin schemes \cite{arnold1983asymptotic,arnold1984asymptotic,costabel1992error,LTTF21}.
Here, we adopt the approach outlined in \cite{LTTF21} by comparing the Fourier symbol of the RBF collocation scheme to that of the Galerkin scheme. Subsequently, we employ a fractional Poincar\'e inequality within the finite-dimensional solution spaces to establish stability. 
The main result of this subsection is presented in \Cref{thm:stability}.
 Our first fundamental result on which the numerical stability is based is the following fractional Poincar\'e  inequality on the finite-dimensional space $V_h$ defined in \cref{eq:FiniteSoluSpace}. 

\begin{proposition}
\label{prop:poincare}
Let $V_h$ be defined by \eqref{eq:FiniteSoluSpace}. 
There exists a constant $C>0$ independent of $h$ such that for any $v\in V_h$, 
\begin{equation}
\label{eq:poincare}
\| v\|_{L^2(\R^d)}^2 \leq C\left(  (-\Delta)^{\frac{\alpha}{2}}  v , v\right)_{L^2(\R^d)} = C | v|_{H^{\alpha/2}(\R^d)}^2. \end{equation}
\end{proposition}

We note that fractional Poincar\'e inequality has been established for functions with compact supports; see \cite{brasco2019note,covi2021unique}. Here, a major challenge in proving \cref{eq:poincare} arises due to the global supports of functions in $V_h$. To the best of our knowledge,  a Poincar\'e type inequality of this nature has not been established previously.
The main idea for proving \cref{eq:poincare} relies on decomposing the $L^2$ integration into high and low frequencies, utilizing the Plancherel theorem: 
\begin{equation*}
\| v\|_{L^2(\R^d)}^2 =\int_{|\xib| < b} |\widehat{v}(\xib)|^2 d\xib + \int_{|\xib| > b} |\widehat{v}(\xib)|^2 d\xib.
\end{equation*}
The high-frequency integration can be easily bounded by $| v|_{H^{\alpha/2}(\R^d)}^2$, while the following lemma establishes a bound for the low-frequency integration. 

 \begin{lemma}
 \label{lem:lowfrequency}
 Let $\gamma = (\varepsilon h)^2$  be fixed.
 There exists a constant $b>0$ depending only on $d$, $\gamma$ and $\Om$ such that
 \beq
 \label{eq:lowfrequency}
\int_{|\xib| < b} |\widehat{u}(\xib)|^2 d\xib \leq \frac{1}{2}\| u\|^2_{L^2(\R^d)}, \quad \ \forall  u \in V_h.
\eeq
\end{lemma}
\begin{proof}
Let $u = \sum_{j=1}^N \lambda_j \vphi^\vep(\xb -\xb_j)$, then
\begin{equation*}
\|  u \|_{L^2(\R^d)}^2  =\sum_{i,j =1}^N \la_i \la_{j}  \int_{\R^d}  \vphi^\varepsilon (\xb -\xb_i )  \vphi^\varepsilon (\xb -\xb_j )  d\xb ,
\end{equation*}
On the other hand, notice that
\beq
\label{eq:FTGaussian}
\widehat{\vphi^\vep}(\xib) = \int_{\R^d} e^{-i \xb\cdot\xib} e^{-\vep^2 |\xb|^2} d\xb = \prod_{k=1}^d \int_{\R} e^{-i x_k\xi_k} e^{-\vep^2 x_k^2} dx_k =\frac{\pi^{d/2}}{\vep^d} e^{-\frac{|\xib|^2}{4\vep^2}}. \eeq
Therefore,
\begin{equation*}
\int_{|\xib| < b} |\widehat{u}(\xib)|^2 d\xib  =  \sum_{i,j =1}^N \la_i \la_{j} \frac{\pi^d}{\vep^{2d}} \int_{|\xib|<b} e^{i (\xb_j - \xb_i)\cdot \xib} e^{-\frac{|\xib|^2}{2\vep^2}} d\xib.
\end{equation*}
Denote 
\begin{align*}
K_1(\xb, \yb) = \int_{\R^d }  \vphi^\varepsilon (\zb -\xb )  \vphi^\varepsilon (\zb -\yb )  d\zb , \text{ and }K_2(\xb, \yb) = \frac{\pi^d}{\vep^{2d}} \int_{|\xib|<b} e^{i (\yb - \xb)\cdot \xib} e^{-\frac{|\xib|^2}{2\vep^2}} d\xib.
\end{align*}
To show \cref{eq:lowfrequency}, we only need to show that $\{  \frac{1}{2}K_1(\xb_i, \xb_j) - K_2(\xb_i, \xb_j) \}_{i,j=1}^N $ is positive semi-definite for a $b>0$ depending only on $d$, $\ga$ and $\Om$. Notice that
\begin{equation*}
\begin{split}
  &K_1(\xb,\yb)=  \int_{\R^d} \vphi^\varepsilon (\zb -\xb )  \vphi^\varepsilon (\zb -\yb )  d\zb = \int_{\R^d} e^{-\varepsilon^2 |\zb-\xb|^2}  e^{-\varepsilon^2 |\zb-\yb|^2}   d\zb\\
    =&\int_{\R^d} e^{-\varepsilon^2 \sum_{i=1}^d (2z_i^2 -2z_ix_i - 2z_iy_i + x_i^2 + y_i^2) } d\zb = \int_{\R^d}  e^{- \varepsilon^2 \sum_{i=1}^d \left[  2\left(  z_i - \frac{x_i+y_i}{2}\right)^2 +  \frac{(x_i-y_i)^2}{2}\right] }  d\zb \\
=&   e^{- \frac{1}{2} \varepsilon^2 |\xb-\yb|^2}   \int_{\R^d}  e^{- 2\varepsilon^2 \left|\zb - \frac{\xb+\yb}{2}\right|^2 } d\zb  = e^{- \frac{1}{2} \varepsilon^2 |\xb-\yb|^2} \int_{\R^d}  e^{- 2\varepsilon^2  \left|\zb  \right|^2 } d\zb \\
=& e^{- \frac{1}{2} \varepsilon^2 |\xb-\yb|^2}\frac{\pi^{d/2}}{(\sqrt{2}\vep)^d}.  
    \end{split}
\end{equation*}
Let $\la_{\min}$ be the minimum eigenvalue of the matrix $G=\{ (2^{-1}\pi)^{d/2} e^{-\frac{1}{2}\varepsilon^2|\xb_i -\xb_j|^2} \}_{i,j=1}^N$, then by \cite[Corollary 12.4]{wendland2004scattered},
\begin{equation*}
\la_{\min} \geq c_1 e^{-\frac{c_2}{\gamma}} \gamma^{-d/2}, 
\end{equation*}
for some $c_1>0$ and $c_2>0$ depending only on $d$. 
On the other hand,
\begin{equation*}
|K_2(\xb_i,\xb_j)| \leq \frac{\pi^d}{\vep^{2d}} \int_{|\xib|<b} e^{-\frac{|\xib|^2}{2\vep^2}} d\xib  = \frac{(b\pi)^d}{\vep^{2d}} \int_{|\xib|<1} e^{-\frac{b^2|\xib|^2}{2\vep^2}} d\xib \leq  \frac{(b\pi)^d h^d}{\vep^{d}\ga^{d/2}} |B_1|,
\end{equation*}
where $|B_1|$ denotes the volume of the unit ball in $\R^d$. Therefore,
\begin{equation*}
\bigg|\sum_{j=1}^N K_2(\xb_i,\xb_j)\bigg|\leq N \frac{(b\pi)^d h^d}{\vep^{d}\ga^{d/2}} |B_1| \leq c_3 \frac{(b\pi)^d}{\vep^{d}\ga^{d/2}} |B_1|,
\end{equation*}
for some $c_3$ depending only on $\Om$. Now choose $b>0$ such that 
\begin{equation*}
c_3 \frac{(b\pi)^d}{\ga^{d/2}} |B_1|  \leq \frac{\la_{\min}}{4} \quad \Longleftrightarrow  \quad b 
\leq \frac{\ga^{1/2}(\la_{\min})^{1/d}}{\pi (4 c_3)^{1/d}},
\end{equation*}
we then have 
\beq
\label{eq:Kij}
\bigg|\sum_{j=1}^N K_2(\xb_i,\xb_j)\bigg|\leq \frac{\la_{\min}}{4 \vep^d}.
\eeq
Notice that since the lower bound of $\lambda_{\min}$ depends only on $d$ and $\gamma$, $b$ only depends on $d$, $\ga$, and $\Om$. 
We claim that 
\begin{equation*}
\left\{ \frac{1}{2}K_1(\xb_i, \xb_j) - K_2(\xb_i,\xb_j)\right\}_{i,j=1}^N \succeq \frac{1}{2\varepsilon^d} \left(G - \la_{\min} I \right) \succeq 0. 
\end{equation*}
Indeed, the above is true since $ \frac{\la_{\min}}{2\varepsilon^d} I  - \{ K_2(\xb_i, \xb_j)\}_{i,j}^N  $ is a positive definite matrix. It is, in fact, a strictly diagonally dominant Hermitian matrix by \cref{eq:Kij}. Therefore we have the desired result.  
\end{proof}

We are now ready to show the proof of the fractional Poincar\'e inequality. 

\begin{proof}[Proof of \Cref{prop:poincare}]
By Plancherel theorem, we have
\begin{equation*}
\| v\|_{L^2(\R^d)}^2 =\int_{\R^d} |\widehat{v}(\xib)|^2 d\xib = \int_{|\xib| < b} |\widehat{v}(\xib)|^2 d\xib + \int_{|\xib| > b} |\widehat{v}(\xib)|^2 d\xib,
\end{equation*}
for $b>0$. Choose $b>0$ depending only on $d$, $\ga$ and $\Om$ such that \cref{eq:lowfrequency} is satisfied. Notice that 
\begin{equation*}
\int_{|\xib| > b} |\widehat{v}(\xib)|^2 d\xib \leq  b^{-\al} \int_{|\xib|>b} |\xib|^{\al} |\widehat{v}(\xib)|^2 d\xib  \leq   b^{-\al} \int_{\R^d} |\xib|^{\al} |\widehat{v}(\xib)|^2 d\xib = b^{-\al} |v|_{H^{\al/2}(\R^d)}^2.
\end{equation*}
Therefore,
\begin{equation*}
\| v\|_{L^2(\R^d)}^2 \leq \frac{1}{2} \| v\|_{L^2(\R^d)}^2 +  b^{-\al} |v|_{H^{\al/2}}^2,
\end{equation*}
which implies the Poincar\'e inequality. 
\end{proof}

In the next, we study the Fourier symbol of the RBF collocation scheme. 
We consider a uniform Cartesian grid on $\R^d$, i.e., $\yb_{\kb} = h\kb$, $\kb\in \Z^d$.
For any continuous function $v:\R^d\to\R$, we define two restriction operators
\begin{equation*}
r^h v := (v(\yb_{\kb}))_{\kb\in \Z^d}, \quad \text{and} \quad r^h_\Om v := (v(\yb_{\kb}))_{\kb\in \Z^d,\yb_{\kb}\in \Om}.  
\end{equation*}
Note that $r^h$ maps $v$ to an infinite sequence and $r^h_\Om$ maps $v$ to a finite sequence. 
The inverse Fourier transform gives us
\begin{equation*}
(-\Delta)^{\al/2} u (\yb_{\kb}) = \frac{1}{(2\pi)^d}\int_{\R^d} |\xib|^\al \widehat{u}(\xib) e^{i \yb_{\kb} \cdot \xib} d\xib.  
\end{equation*}
Therefore 
\begin{equation*}
\left[ (-\Delta)^{\al/2} \vphi^\vep(\cdot - \yb_{\kb'}) \right]\Big|_{\yb_{\kb}} = \frac{1}{(2\pi)^d}\int_{\R^d} |\xib|^\al \widehat{\vphi^\vep}(\xib) e^{i (\yb_{\kb}-\yb_{\kb'})\cdot\xib} d\xib, 
\end{equation*}
since the Fourier transform of $\vphi^\vep(\cdot - \yb_{\kb'})$ is $e^{-i\yb_{\kb'}\cdot\xib} \widehat{\vphi^\vep}(\xib)$.  
Now for any $(v_{\kb}) := (v_{\kb})_{\kb\in\Z} \in l^2(\Z^d)$, we define the associated Fourier series
\begin{equation*}
\widetilde{v} (\xib ) = \sum_{\kb\in \Z^d} e^{i\kb\cdot\xib } v_{\kb}, \quad \ \forall \xib \in  (-\pi ,\pi)^d. 
 \end{equation*}
 We equip $l^2(\Z^d)$ with a scaled $l^2$ inner product and the associated norm by
 \begin{equation*}
 ((u_{\kb}), (v_{\kb}))_{l^2} = h^d \sum_{\kb\in \Z^d} u_{\kb}v_{\kb}\quad \text{and}\quad \|(v_{\kb})\|_{l^2} = \sqrt{h^d \sum_{\kb\in \Z^d} (v_{\kb})^2}.
 \end{equation*}
 Therefore for any $v = \sum_{\kb\in\Z^d} \lambda_{\kb} \vphi^\vep(\cdot-\yb_{\kb}) $ and $(\la_{\kb}):=(\la_{\kb})_{\kb\in\Z^d}\in l^2(\Z^d)$,
\begin{equation*}
\begin{split}
&((\la_{\kb}),  r^h (-\Delta)^{\frac{\alpha}{2}} v ))_{l^2} = \frac{h^d}{(2\pi)^d}\sum_{\kb, \kb'\in \Z^d} \la_{\kb} \la_{\kb'} \int_{\R^d}  |\xib|^\al \widehat{\vphi^\vep}(\xib) e^{i (\yb_{\kb}-\yb_{\kb'})\cdot \xib} d\xib  \\
&= \frac{h^d}{(2\pi)^d}\sum_{\kb, \kb'\in \Z^d} \la_{\kb} \la_{\kb'}  \sum_{\jb\in\Z^d } \int_{\left(-\frac{\pi}{h},\frac{\pi}{h}\right)^d + \frac{2\pi}{h}\jb}  |\xib|^\al \widehat{\vphi^\vep}(\xib) e^{i h(\kb- \kb')\cdot \xib} d\xib  \\
&=  \frac{1}{(2\pi)^d}\sum_{\kb, \kb'\in \Z^d} \la_{\kb} \la_{\kb'}  \sum_{\jb\in\Z^d } \int_{(-\pi, \pi)^d}  \Big|\frac{\xib - 2\pi\jb}{h} \Big|^\al \widehat{\vphi^\vep}\Big(\frac{\xib - 2\pi\jb}{h}\Big) e^{i (\kb- \kb')\cdot\xib} d\xib  \\
&= \frac{1}{(2\pi)^d}\int_{(-\pi,\pi)^d} \big|\widetilde{\la}(\xib )\big|^2\cE_C(h, \xib) d\xib, 
\end{split} 
\end{equation*}
where 
\begin{equation}
\label{eq:ECollocation}
\cE_C(h, \xib): =  \sum_{\jb\in\Z^d }  \left|\frac{\xib - 2\pi\jb}{h} \right|^\al \widehat{\vphi^\vep}\left(\frac{\xib - 2\pi\jb}{h}\right).  
\end{equation}
We note that $\cE_C(h, \xib)$ represents the Fourier symbol for the RBF collocation scheme. 
Similarly, by the Plancherel theorem
\begin{equation*}
\begin{split}
&\left(  (-\Delta)^{\frac{\alpha}{2}}  v , v\right)_{L^2} = \sum_{\kb, \kb'\in \Z^d}\la_{\kb} \la_{\kb'}  \left( (-\Delta)^{\frac{\alpha}{2}}  \vphi^\vep(\cdot - \yb_{\kb}),\  \vphi^\vep(\cdot - \yb_{\kb'}) \right)_{L^2} \\
&=\frac{1}{(2\pi)^d}  \sum_{\kb, \kb'\in \Z^d}\la_{\kb} \la_{\kb'}   \int_{\R^d}   |\xib|^\al \big|\widehat{\vphi^\vep}(\xib) \big|^2 e^{i (\yb_{\kb}-\yb_{\kb'})\cdot\xib} d\xib  \\
&= \frac{h^{-d}}{(2\pi)^d}\sum_{\kb, \kb'\in \Z^d} \la_{\kb} \la_{\kb'}  \sum_{\jb\in\Z^d} \int_{(-\pi,\pi)^d}   \Big|\frac{\xib - 2\pi\jb}{h} \Big|^\al \Big|\widehat{\vphi^\vep}\Big(\frac{\xib - 2\pi\jb}{h}\Big)\Big|^2 e^{i (\kb- \kb')\cdot \xib} d\xib \\
&=\frac{1}{(2\pi)^d}\int_{(-\pi,\pi)^d}  \big|\widetilde{\la}(\xib ) \big|^2 \cE_G(h, \xib) d\xib,  \end{split} 
\end{equation*} 
where 
\beq
\label{eq:EGalerkin}
\cE_G(h, \xib): =  h^{-d}\sum_{\jb\in\Z^d }  \Big|\frac{\xib - 2\pi\jb}{h} \Big|^\al\Big|\widehat{\vphi^\vep}\Big(\frac{\xib - 2\pi\jb}{h}\Big)\Big|^2 ,
\eeq
and $\cE_G(h, \xib)$ represents the Fourier symbol for the Galerkin scheme.  
The following lemma provides a comparison between $\cE_C(h, \xib)$ and $\cE_G(h, \xib)$.
\begin{lemma}
\label{lem:comparison}
Let $\cE_C(h, \xib)$ and $\cE_G(h, \xib)$ be defined by \cref{eq:ECollocation,eq:EGalerkin}, respectively. Then
\begin{equation*}
\cE_C(h, \xib) \geq (\ga/\pi)^{d/2}\cE_G(h, \xib),
\end{equation*}
where $\ga=(\vep h)^2$.
\end{lemma}
\begin{proof}
By \cref{eq:FTGaussian}, $h^{-d}\widehat{\vphi^\vep}(\xib) = h^{-d}|\widehat{\vphi^\vep}(\xib)| \leq \pi^{d/2}/(\vep h)^d = (\pi/\ga)^{d/2} $.
Therefore, the inequality is implied from \cref{eq:ECollocation} and \cref{eq:EGalerkin}.
\end{proof}

Additionally, we establish an equivalent norm on $l^2(\Z^d)$. 
For any $(\la_{\kb})\in l^2(\Z^d)$, we define another norm $\| (\la_{\kb})\|_{\vphi}$ by 
\begin{equation*}
\| (\la_{\kb})\|_{\vphi} := \bigg\| \sum_{\kb\in \Z^d} \la_{\kb} \vphi^{\vep}(\cdot - \yb_{\kb})\bigg\|_{L^2(\R^d)}. 
\end{equation*}
Now we show that $\| (\la_{\kb})\|_{l^2}$ and $\| (\la_{\kb})\|_{\vphi}$ are equivalent. 
\begin{lemma}
\label{lem:equivalence_norms}
Let $\ga=(\vep h)^2$ be fixed. 
 $\| \cdot\|_{l^2}$  and   $\| \cdot\|_{\vphi}$ are equivalent, namely, there exist two constants $C_1>0$ and $C_2>0$ independent of $h$ such that
 \begin{equation*}
 C_1 \| (\la_{\kb})\|_{l^2} \leq \| (\la_{\kb})\|_{\vphi} \leq C_2 \| (\la_{\kb})\|_{l^2}, \quad \ \forall (\la_{\kb})\in l^2(\Z^d). 
 \end{equation*}
\end{lemma}
\begin{proof}
By Parseval's identity, we have 
\begin{equation*}
\| (\la_{\kb})\|^2_{l^2} = h^d \sum_{\kb\in\Z^d} \la_{\kb}^2 = \frac{h^d}{(2\pi)^d}\int_{(-\pi, \pi)^d} \big|\widetilde{\la}(\xib)\big|^2 d\xib . 
\end{equation*}
On the other hand, similar to the argument made previously, we have
\begin{equation*}
\begin{split}
\| (\la_{\kb})\|^2_{\vphi} = &\sum_{\kb, \kb'\in \Z^d}\la_{\kb} \la_{\kb'}  \left( \vphi^\vep(\cdot - \yb_{\kb}) , \vphi^\vep(\cdot - \yb_{\kb'}) \right)_{L^2}\\
=&\frac{h^d}{(2\pi)^d}\int_{(-\pi,\pi)^d}  \big|\widetilde{\la}(\xib ) \big|^2 \cE_\vphi(h, \xib) d\xib,
\end{split}
\end{equation*}
where $\cE_\vphi(h, \xib)  = h^{-2d}\sum_{\jb\in\Z^d }  \left|\widehat{\vphi^\vep}\left(\frac{\xib - 2\pi\jb}{h}\right)\right|^2 $. Notice that by \cref{eq:FTGaussian}, 
\begin{equation*}
\Big|\widehat{\vphi^\vep}\Big(\frac{\xib - 2\pi\jb}{h}\Big) \Big|^2= \frac{\pi^d}{\vep^{2d}} e^{-\frac{|\xib - 2\pi \jb |^2}{2 (\vep h)^2}}.
\end{equation*}
Therefore $\cE_\vphi(h, \xib)$ is independent of $h$ and 
\begin{equation*}
\cE_\vphi(h, \xib) = \cE_\vphi(\xib):=(\pi/\ga)^d\sum_{\jb\in\Z^d} e^{-\frac{|\xib - 2\pi \jb |^2}{2\ga}}. 
\end{equation*}
It is easy to see that for $\xib\in (-\pi, \pi)^d$, there exists $C_1>0$ and $C_2>0$ such that
\begin{equation*}
C_1 \leq \cE_\vphi(\xib) \leq C_2. 
\end{equation*}
The equivalence of norms is thus shown.  
\end{proof}

Finally, we establish the numerical stability of the Gaussian RBF collocation scheme. The main idea of the proof follows \cite[Theorem 4.1]{LTTF21}.  
\begin{theorem}[numerical stability]
\label{thm:stability}
Let $\ga = (\vep h)^2$ be fixed. For any $v\in V_h$, we have \begin{equation*}
\|v\|_{L^2(\R^d)} \leq C \| r^h_{\Om} (-\Delta)^{\frac{\alpha}{2}} v\|_{l^2}, \end{equation*}
where $C>0$ is a constant independent of $h$ and $\vep$.
\end{theorem}
\begin{proof}
For any $v\in V_h$, one can write 
\begin{equation*}
v(\xb) = \sum_{\kb\in\Z^d} \la_{\kb} \vphi^\vep(\xb - \yb_{\kb} )
\end{equation*}
where $\la_{\kb}=0$ if $\yb_{\kb}\notin\Om$. By Cauchy-Schwartz inequality in $l^2(\R^d)$, \Cref{lem:equivalence_norms} and \Cref{lem:comparison}, we obtain
\begin{equation*}
\begin{split}
\|v \|_{L^2(\R^d)} \cdot \| r^h_{\Om} (-\Delta)^{\frac{\alpha}{2}} v\|_{l^2} = &
\|(\la_{\kb}) \|_{\vphi} \cdot \| r^h_{\Om} (-\Delta)^{\frac{\alpha}{2}} v\|_{l^2} \\
\geq &C ((\la_{\kb}), r^h_{\Om} (-\Delta)^{\frac{\alpha}{2}} v)_{l^2}\\
= & C ((\la_{\kb}), r^h (-\Delta)^{\frac{\alpha}{2}} v)_{l^2}  \\
\geq & C' \left(  (-\Delta)^{\frac{\alpha}{2}}  v , v\right)_{L^2}.
\end{split}
\end{equation*}
Now invoking the Poincar\'e inequality established in \Cref{prop:poincare}, we obtain the desired result. 
\end{proof}

 \subsection{Convergence of the numerical scheme}
\label{subsection-convergence}
We present the convergence analysis of the Gaussian RBF collocation method introduced in \Cref{section-method}. 
\begin{theorem}[convergence]\label{them-conv}
Let $u$ be the exact solution and $u_h$ be the numerical solution. For $h\in (0,1)$, the following estimate holds. 
\begin{equation*}
 \| u - u_h\|_{L^2(\Om)} \leq C\bigg(h^{M-2} |u|_{M} +\sum_{|\bm\beta|\in\{ 0,2\}}\sup_{\xb\in\R^d}\sum_{|\bm\al|=0}^{M-1}\frac{|a^{(\bm\beta)}_{\bm\al}(\xb/h)|}{\bm\al !} h^{|\bm\al|-|\bm\beta|} |D^{\bm\al} u(\xb)| \bigg).
\end{equation*}
\end{theorem}
\begin{proof}
Notice that for any $\yb_{\kb} \in \Om$,
\begin{equation*}
(-\Delta)^{\frac{\al}{2}} u (\yb_{\kb}) = f(\yb_{\kb}) = (-\Delta)^{\frac{\al}{2}} u_h (\yb_{\kb}). 
\end{equation*}
In other words, $r^h_{\Om} (-\Delta)^{\frac{\al}{2}} u = r^h_{\Om} (-\Delta)^{\frac{\al}{2}} u_h$. 
Using this fact and \Cref{thm:stability}, we obtain
\begin{equation*}
\|\cI_h u - u_h \|_{L^2(\R^d)} \leq C \| r^h_\Om (-\Delta)^{\frac{\al}{2}} (\cI_h u - u_h)\|_{l^2} 
= C \| r^h_\Om (-\Delta)^{\frac{\al}{2}} \cI_h u - r^h_\Om (-\Delta)^{\frac{\al}{2}}  u\|_{l^2}.
\end{equation*}
 For any $\yb_{\kb} \in \Om$, we use
\Cref{thm:consistency} to get a bound of $|(-\Delta)^{\frac{\al}{2}} (\cI_h u - u)(\yb_{\kb})|$. Therefore, 
\begin{equation*}
\begin{split}
&\|\cI_h u - u_h \|_{L^2(\R^d)} \leq \, C\| r^h_\Om (-\Delta)^{\frac{\al}{2}} \cI_h u - r^h_\Om (-\Delta)^{\frac{\al}{2}}  u\|_{l^2} \\
\leq & \,
C' \Bigg(h^{M-2} |u|_{M} +\sum_{|\bm\beta|\in\{ 0,2\}}\sup_{\xb\in\R^d}\sum_{|\bm\al|=0}^{M-1}\frac{|a^{(\bm\beta)}_{\bm\al}(\xb/h)|}{\bm\al !} h^{|\bm\al|-|\bm\beta|} |D^{\bm\al} u(\xb)| \Bigg).
\end{split}
\end{equation*}
On the other hand, using \Cref{lem:approx_err}
\begin{equation*}
\begin{split}
\|  u -\cI_h u\|_{L^2(\Om)}\leq &\, \sqrt{|\Om|}\|  u -\cI_h u\|_{L^\infty(\Om)} \\
\leq &\, C \Bigg(h^M |u|_{M} + \sup_{\xb\in\R^d}\sum_{|\bm\al|=0}^{M-1}\frac{|a^{(0)}_{\bm\al}(\xb/h)|}{\bm\al !} h^{|\bm\al|} |D^{\bm\al} u(\xb)|\Bigg).
\end{split}
\end{equation*}
Combining the above estimates above and the triangle inequality $
\| u -u_h\|_{L^2(\Om)} \leq \| u -\cI_h u \|_{L^2(\Om)} + \| \cI_h u -u_h\|_{L^2(\Om)}$, we obtain the desired estimate. 
\end{proof}

From the above theorem, we notice that the saturation error depends on the coefficient  $ |a^{(\bm\beta)}_{\bm\al}(\xb)|/\bm\al!$ for $|\bm\beta|\in\{0,2\}$. Since obtaining a direct estimate of this coefficient is challenging, we resort to numerical approximations, as detailed in \Cref{appendix}. The numerical results demonstrate that the saturation error tends to increase as $\gamma$ grows. In addition, within a specific range of $\gamma$, one can effectively manage the saturation error to remain small.

\section{Numerical experiments}
\label{section-experiments} 

To further evaluate our method, we numerically solve the fractional Poisson equation (\ref{Poisson}) with homogeneous Dirichlet boundary conditions, i.e. $u(\bx) = 0$ for $\bx \in \Og^c$.
For notational simplicity, we denote $c^* = \sqrt{\gamma}=\vep h$ here and in the subsequent sections. 
Hence,  the shape parameter $\varepsilon = c^*/h$ increases as $h$ decreases, if $c^*$ is a fixed constant.  
The root mean square (RMS) error in the solution of the Poisson equation \cref{Poisson} is defined as 
\begin{equation*}
\|e_u\|_{\rm rms} = \bigg(\fl{1}{\bar{M}}\sum_{k = 1}^{\bar{M}}\big|u(\xb_k) - u_h(\xb_k)\big|^2\bigg)^{\fl{1}{2}},
\end{equation*}
where $u$ and $u_h$ represent the exact and numerical solutions, respectively, and 
$\bar{M} \gg N$ denotes the total number of interpolation points on $\bar{\Og}$. In the following, we present 1D and 2D numerical examples for solving \cref{Poisson}.
\bigskip 

\noindent{\bf Example 1.}
Consider the 1D fractional Poisson equation \cref{Poisson} with $g\equiv 0$, $\Omega = (-1, 1)$, and the right-hand side function
\begin{equation}\label{ex1-f}
f(x) = \frac{2^\ap\Gamma\big((\ap+1)/2\big)\Gamma\big(s+1\big)}{\sqrt{\pi}\Gamma\big(s+1-\ap/2\big)}\,_2F_1\Big(\frac{\ap+1}{2}, -s+\frac{\ap}{2}; \frac{1}{2}; x^2\Big),  
\end{equation}
where constant $s > 0$, and $_2F_1$ represents the Gauss hypergeometric function. 
Then the exact solution of \cref{Poisson} with \cref{ex1-f} is given by $u(x) = [(1-x^2)_+]^s$ for $s > 0$. 

In \cref{Tab4-2-1}, we present numerical errors $\|e_u\|_{\rm rms}$ and condition numbers ${\mathcal K}$ for different $\alpha$ and number of points $N$, where the RBF center points are uniformly spaced on $(-1, 1)$ with $h = 2/(N+1)$. It shows that numerical errors reduce quickly as the number of points $N$ increases.  
\begin{table}[htbp]
\begin{center} 
\begin{tabular}{|c|c|c|c|c|c|c|}
\hline
 & \multicolumn{2}{|c|}{$\ap = 0.4$} & \multicolumn{2}{|c|}{$\ap = 1$} & \multicolumn{2}{|c|}{$\ap = 1.5$} \\
\cline{2-7}
$N$     & $\|e_u\|_{\rm rms}$ & $\mathcal{K}$ & $\|e_u\|_{\rm rms}$ & $\mathcal{K}$ & $\|e_u\|_{\rm rms}$ & $\mathcal{K}$ \\
\hline
 7 & 1.971e-3  & 288.61  & 5.773e-3  &141.17 &  2.612e-2 &  82.194  \\
15 & 3.812e-4 &  1586.6 &  1.066e-3 &  627.15 &  1.583e-3 &  331.69 \\
31 & 2.509e-5 &  2938.4 &  8.403e-5& 1086.0 & 1.708e-4 &  551.78 \\
63 & 1.273e-6  & 3447.9 & 4.773e-6  & 1258.4  & 1.146e-5  & 632.09 \\
127 & 5.899e-8 &  3584.0 &  2.431e-7&  1304.5  & 6.725e-7  & 653.69 \\
255 & 2.616e-9  & 3617.8  & 1.181e-8 &  1315.9  & 3.740e-8 & 659.11  \\
511 & 1.12e-10 &  3626.1 & 5.59e-10  & 1318.8  & 2.027e-9  & 982.86 \\
\hline
\end{tabular}
\caption{RMS errors $\|e_u\|_{\rm rms}$ and condition numbers $\mathcal K$ in solving Poisson problem \cref{Poisson} with exact solution $u(x) = [(1-x^2)_+]^4$,  where  $c^* = 1/2$  and thus $\vep = (N+1)/4$.}\label{Tab4-2-1}
\end{center}
\end{table}
Compared to the method in \cite{Burkardt2021}, the condition numbers are significantly smaller, 
suggesting the effectiveness of our method in addressing the ill-conditioning issue of RBF-based methods. 
\Cref{Fig-Ex1-1} further compares the numerical errors for different $s$. 
The larger the value of $s$,  the smoother the function $u$ around $\p\Og$, the smaller the numerical errors. 
\begin{figure}[htb!]
\centerline{
\includegraphics[height = 4.86cm, width = 6.56cm]{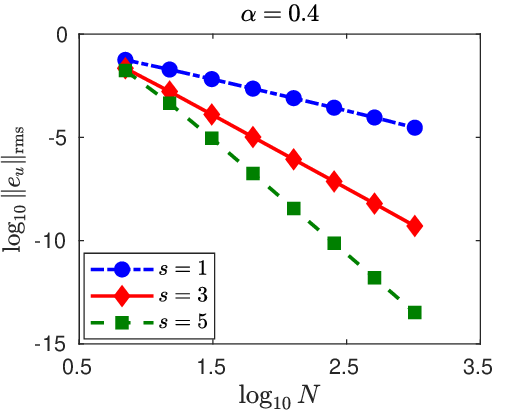}\hspace{-4mm}
\includegraphics[height = 4.86cm, width = 6.56cm]{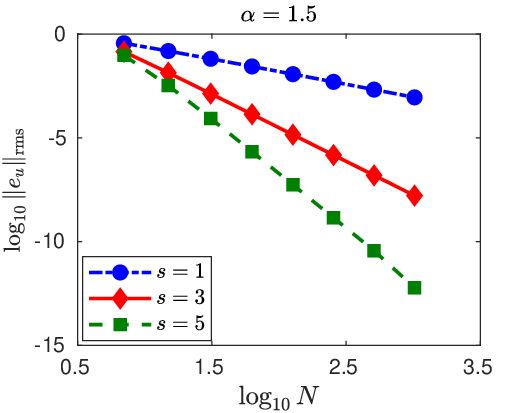}
}
\caption{Numerical errors in solving Poisson problem (\ref{Poisson}) with exact solution $u(x) = [(1-x^2)_+]^s$,  where  $c^* = 1/2$  and thus $\vep = (N+1)/4$.}
\label{Fig-Ex1-1}
\end{figure}
From \cref{Tab4-2-1} and \cref{Fig-Ex1-1}, we find that numerical errors $\|e_u\|_{\rm rms}$ decrease at a order of $O(h^p)$ with $p > s$.  
 
We set $c^* = 1/2$ in the above studies.  
In \cref{Fig-Ex1-2}, we further explore the effects of $c^*$ on numerical errors and condition numbers. 
\begin{figure}[htb!]
\centerline{
(a)\includegraphics[height = 4.86cm, width = 6.56cm]{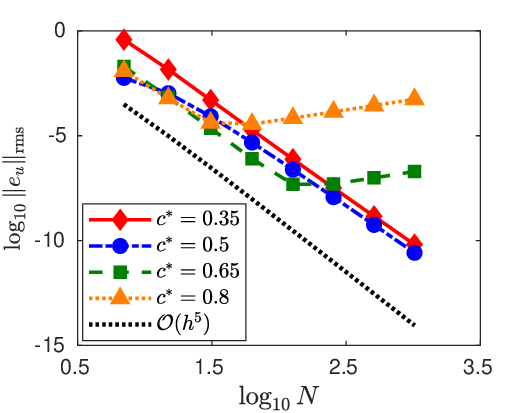}\hspace{-6mm}
(b)\includegraphics[height = 4.86cm, width = 6.56cm]{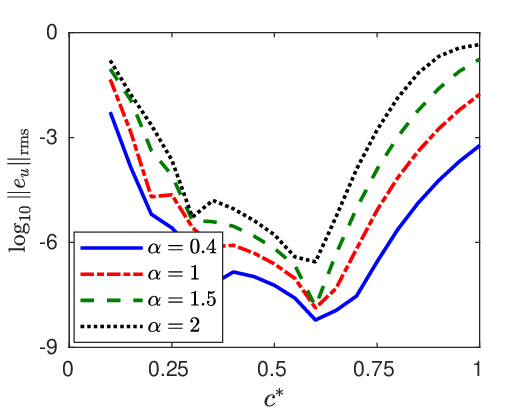}
}
\centerline{
(c)\includegraphics[height = 4.86cm, width = 6.56cm]{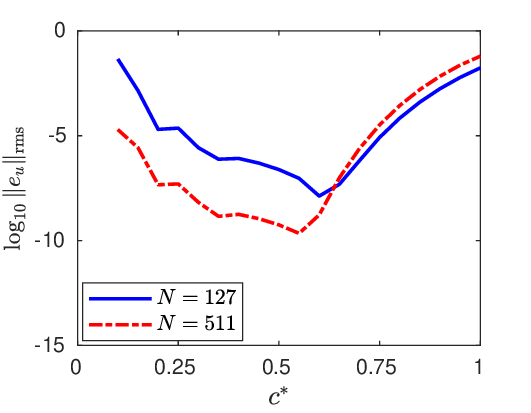}\hspace{-6mm}
(d)\includegraphics[height = 4.86cm, width = 6.56cm]{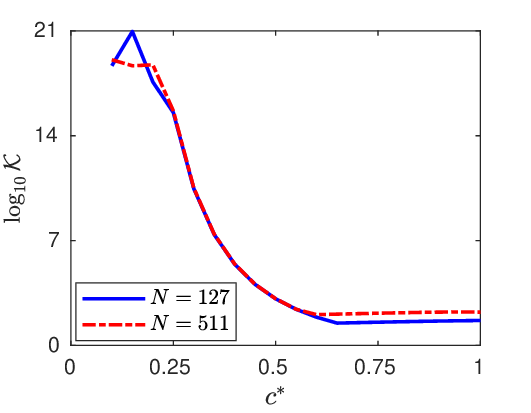}
}
\caption{Comparison of numerical errors for different $c^*$ in solving problem (\ref{Poisson}) with exact solution $u(x) = [(1-x^2)_+]^4$, where $\ap = 1$ in (a)--(c), and $N = 127$ in (d).}
\label{Fig-Ex1-2}
\end{figure}
\Cref{Fig-Ex1-2} (a) shows that the saturation error becomes dominant if the number of points $N > N_{\rm cr}$, where $N_{\rm cr}$ is a threshold value that depends on $c^\ast$.
The larger the value of $c^*$, the smaller the threshold $N_{\rm cr}$, and the bigger the saturation error, consistent with our analysis of the saturation error in \Cref{appendix}. 
Moreover, the choice of $c^*$ shows insignificant dependence on the power $\ap$; see \cref{Fig-Ex1-2} (b). 
From \cref{Fig-Ex1-2} (c) \& (d), we find that numerical errors are large when $c^*$ is either too small or too large. 
Small constant $c^*$ leads to large condition number, introducing large errors in solving resulting system. 
Choosing large $c^*$ can greatly reduce the conditional number, but at the same time it also presents large saturation errors. 
Hence, the constant $c^*$ should be chosen from a proper range to balance the saturation errors and conditional numbers. 

\bigskip
\noindent{\bf Example 2.} Consider the 2D fractional Poisson equation \cref{Poisson} with $g\equiv0$ and a disk domain $\Og = \{\bx \,|\, x^2 + y^2 < 1\}$. 
The function $f$ in \cref{Poisson} is chosen as
\begin{equation}\label{2Df0}
f(x, y) = \frac{\alpha\,2^{\ap-1}\Gamma(\alpha/2)(4!)}{\Gamma(5-\alpha/2)}\,_2F_1\Big(\frac{\alpha}{2}+1,\,\frac{\alpha}{2}-4;\, 1;\,  (x^2+y^2)^2\Big),
\end{equation}
for $(x, y)\in \Og$. 
In this case, the exact solution is a compact support function on the unit disk, i.e. $u(\xb) = [(1-|\xb|^2)_+]^4$. 
The RBF center points are chosen as uniform grid points within the domain $\Og$ with $h = h_x = h_y$.

\Cref{Tab-Ex3} presents numerical errors $\|e_u\|_{\rm rms}$ for different $\ap$ and $h$, where constant $c^* = 1/2$ is fixed.  
\begin{table}[htbp]
\begin{center} 
\begin{tabular}{|c|c|c|c|c|c|c|c|}
\hline
 &  & \multicolumn{2}{|c|}{$\ap = 0.4$} & \multicolumn{2}{|c|}{$\ap = 1$} & \multicolumn{2}{|c|}{$\ap = 1.5$} \\
\cline{3-8}
$N$ & $h$ & $\|e_u\|_{\rm rms}$ & $\mathcal{K}$ & $\|e_u\|_{\rm rms}$ & $\mathcal{K}$ & $\|e_u\|_{\rm rms}$ & $\mathcal{K}$ \\
\hline
45  & $1/4$ &5.898e-3 &1.207e5 &2.288e-2 &4.427e4 &4.975e-2 &2.321e4\\
193 & $1/8$ &3.370e-4 &3.855e6 &1.299e-3 &1.245e6 &3.405e-3 &5.684e5\\
793 & $1/16$&1.497e-5 &1.774e7 &5.504e-5 &5.374e6 &1.757e-4 &2.307e6\\
3205& $1/32$&7.405e-7 &2.696e7 &2.377e-6 &8.001e6 &6.441e-6 &3.382e6\\
\hline
\end{tabular}
\caption{RMS errors $\|e_u\|_{\rm rms}$  in solving the two-dimensional Poisson equation (\ref{Poisson}) with (\ref{2Df0}),  where constant $c^* = 1/2$ is fixed.}\label{Tab-Ex3}
\end{center}
\vspace{-.3cm}
\end{table}
The observations here are similar to those in Example 1 -- numerical errors decrease with a rate of $O(h^p)$ for $p > 4$. 
Moreover, \Cref{Fig-Ex2-1} compares numerical errors for different $c^*$ and $N$. 
Similar to one-dimensional cases, the saturation errors become dominant if $N$ becomes too large, i.e., $N < N_{\rm cr}$ for some threshold value $N_{\rm cr}$.
We consistently observe that $N_{\rm cr}$ decreases as $c^\ast$ increases.
\begin{figure}[htb!]
\centerline{
(a)\includegraphics[height = 4.86cm, width = 6.56cm]{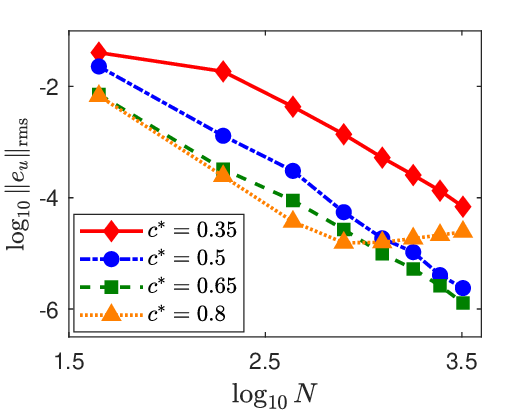}\hspace{-6mm}
(b)\includegraphics[height = 4.86cm, width = 6.56cm]{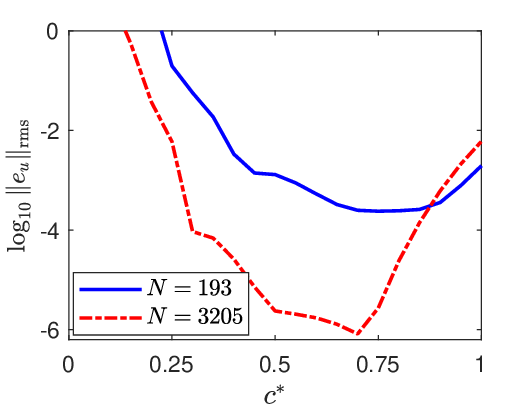}}
\caption{Comparison of numerical errors for different  $c^*$ in solving problem (\ref{Poisson}) with $\ap = 1$ on a unit disk, where the exact solution $u(\xb) = [\big(1-|\xb|^2\big)_+]^4$.}
\label{Fig-Ex2-1}
\vspace{-.3cm}
\end{figure}
Comparing \cref{Fig-Ex1-2} and \cref{Fig-Ex2-1}, we find that regardless of the dimension $d$, the critical distance $h_{\rm cr}$ -- where the saturation error becomes dominant -- remains consistent for a given $c^*$. 

\bigskip 
\noindent{\bf Example 3.} In this example, we solve the 2D Poisson problem \cref{Poisson} with $g\equiv 0$ on a square domain $\Og = (-2, 2)^2$. 
The function $f$ in (\ref{Poisson}) is chosen as 
\begin{equation}\label{2Df}
f(x, y) = 6^\ap\Gamma\big(2+\frac{\ap}{2}\big) \,_1F_1\Big(2+\frac{\ap}{2};\, 2;\, -9|\xb|^2\Big) y. 
\end{equation}
The exact solution of (\ref{Poisson}) with (\ref{2Df}) can be  approximated by $u(x, y) = y e^{-9(x^2+y^2)}$ for $(x, y) \in \Og$. 
In our simulations, the RBF center points are chosen as equally spaced points with $h = h_x = h_y = 4/(N_x+1)$. 
Hence, the total number of points is $N = N_x^2$ with $N_x$ the number of points in $x$-direction.

\Cref{Tab4-2-2} presents numerical errors $\|e_u\|_{\rm rms}$  for different $\alpha$, where $c^* = 1/2$ is fixed. 
The solution in this case is much smoother than that in Example 2. 
\begin{table}[htb!]
\begin{center} 
\begin{tabular}{|c|c|c|c|c|c|c|c|}
\hline
& & \multicolumn{2}{|c|}{$\ap = 0.4$} & \multicolumn{2}{|c|}{$\ap = 1$} & \multicolumn{2}{|c|}{$\ap = 1.5$} 
 \\
\cline{3-8}
$N$   & $h$  & $\|e_u\|_{\rm rms}$ & $\mathcal{K}$ & $\|e_u\|_{\rm rms}$ & $\mathcal{K}$ & $\|e_u\|_{\rm rms}$ & $\mathcal{K}$ \\
\hline
$7^2$ & $1/2$  & 2.546e-3  &5.861e6  & 4.702e-3 &1.849e6 & 7.718e-3 &8.227e5  \\
$15^2$ & $1/4$ & 3.897e-9  &2.011e7 & 4.816e-9  &6.034e6 & 6.509e-9 & 2.573e6   \\
$31^2$ & $1/8$ & 5.006e-16 &2.765e7 & 1.335e-15 &8.196e6 & 3.355e-15 &3.461e6 \\
\hline
\end{tabular}
\caption{RMS errors $\|e_u\|_{\rm rms}$  in solving the two-dimensional Poisson equation in Example 3,  where constant $c^* = 1/2$ is fixed.}\label{Tab4-2-2}
\end{center}
\vspace{-.5cm}
\end{table}
Table \ref{Tab4-2-2} shows that our method can achieve a spectral accuracy if the solution is smooth enough. 
Numerical errors become negligible if the distance $h$ reduces to $h = 1/8$, while the condition number remains around $10^6\sim 10^7$.   

\section{Nonhomogeneous boundary value problems}
\label{section-nonhomogeneous}
In this section, we generalize our method to treat nonhomogeneous Dirichlet boundary conditions. 
To this end, we propose a two-stage RBF method. 

Following the discussions in \Cref{section-method}, we introduce an auxiliary function $w(\xb)$ for $\xb \in {\mathbb R}^d$, such that 
\begin{equation}\label{fun-w}
w(\xb) = 
\left\{ 
\begin{aligned}
w_h(\xb), \qquad &\text{ if } \xb \in \Om, \\
g(\xb),\qquad     &\text{ if } \xb\in \Om^c.
\end{aligned}
\right.
\end{equation}
There are two things to note for the construction of $w$. First, since the regularity of solutions affects convergence rates, as illustrated in \Cref{section-analysis} and the numerical examples in \Cref{section-experiments}, it is preferred that $w$ be defined as regular as possible. Second, since the right-hand side of \cref{Poisson-v} involves the evaluation of $(-\Dt)^\fl{\ap}{2}w$, 
we want to construct $w$ such that the action of fractional Laplacian on it can be easily approximated. 
We again use Gaussian RBFs to present $w_h$. 
More precisely, let
\begin{equation}\label{what}
w_h(\xb) := \sum_{l = 1}^M\widetilde{\lambda}_l\,e^{-\widetilde{\varepsilon}^2|\xb - \widetilde{\xb}_l|^2}, \qquad \mbox{for} \ \ \xb \in {\mathbb R}^d.
\end{equation}
The selection of shape parameter $\widetilde{\varepsilon}$ and RBF center points $\widetilde{\xb}_l$ in \cref{what} is independent of those in \cref{ansatz}, and we usually choose the center points $\widetilde{\xb}_l$ near the boundary region. 
To determine coefficients $\widetilde{\lambda}_l$, we require  
\begin{equation}\label{w-BC}
w_h(\widetilde{\xb}_k') = g(\widetilde{\xb}'_k), \qquad \mbox{for} \ \ \widetilde{\xb}_k' \in \widetilde{\Og},
\end{equation}
 where $\widetilde{\Og}$ denotes a closed region satisfying $\widetilde{\Og}\subset \Og^c$ and $\overline{\Og} \cap \widetilde{\Og} = \p\Og$, and $\{\widetilde{\xb}_k'\}_{k=1}^M \subset \widetilde{\Og}$ is a set of test points. 
In practice, the shape parameter and RBF center points are chosen such that the error between $w_h$ and $g$ over $\widetilde{\Og}$ is smaller than a desired tolerance $\tau$, i.e. $\|w_h - g\|_{\rm rms} < \tau$. 
Hence, we have  $w_h(\xb) \approx g(\xb)$ for any $\xb \in \widetilde{\Og}$.

Next, we approximate the function $(-\Dt)^\fl{\ap}{2}w(\xb)$ in \cref{Poisson-v}. 
Substituting \cref{fun-w} into the definition  \cref{integralFL}, we obtain
\begin{eqnarray}\label{Lapw}
&&(-\Delta)_h^{\frac{\al}{2}}w (\xb) = C_{d,\al} \int_{\R^d} \frac{w_h(\xb)-w_h(\yb)}{|\yb-\xb|^{d+\al}} d\yb + C_{d,\al} \int_{\Om^c} \frac{w_h(\yb) - g(\yb)}{|\yb-\xb|^{d+\al}} d\yb \nn\\
&&\hspace{2cm} = (-\Dt)^\fl{\ap}{2}w_h(\xb) + C_{d,\al} \int_{\Om^c\backslash\widetilde{\Og}} \frac{w_h(\yb) - g(\yb)}{|\yb-\xb|^{d+\al}} d\yb, \qquad \mbox{for} \ \ \xb \in \Og. \qquad
\end{eqnarray}
The integral region in the second term reduces to $\Og^c\backslash\widetilde{\Og}$ because of (\ref{w-BC}). 
It is clear that the second term in (\ref{Lapw}) is caused by the nonhomogeneous boundary conditions. 
Since $\xb \in \Og$ and $\yb \in \Og^c\backslash\widetilde{\Og}$, the integral is free of singularity and thus can be accurately approximated by numerical quadrature rules. 
While the term $(-\Dt)^\fl{\ap}{2}w_h(\xb)$ can be analytically obtained by combining (\ref{what}) and (\ref{analy}). 

Hence, our two-stage RBF method for solving nonhomogeneous boundary value problems includes:  (i) constructing $w_h(\xb)$ and approximating $(-\Dt)^\fl{\ap}{2}w$ with RBFs as in (\ref{Lapw}); (ii) solving the homogeneous Poisson equation (\ref{Poisson-v}) with RBF methods in Section \ref{section-method}. 
Next, we will present two examples to test the performance of our method in solving the problem (\ref{Poisson}) with nonhomogeneous boundary conditions. 

\bigskip
\noindent{\bf Example 4.}  Consider the 1D fractional Poisson equation (\ref{Poisson}) on $\Og = (-1, 1)$, where the functions $f$ and $g$ are chosen as
\begin{eqnarray}
&&f(x) = \frac{2^\ap\Gamma\big((1+\ap)/2\big)\Gamma\big(1+\ap/2\big)}{\sqrt{\pi}}\,_2F_1\Big(\frac{1+\ap}{2}, 1+\fl{\ap}{2};\, \frac{1}{2}; -x^2\Big), \nonumber\\
&&g(x) = \frac{1}{1+x^2}.\nonumber
\end{eqnarray}
Then exact solution of (\ref{Poisson}) is given by $u(x) = 1/(1+x^2)$ for $x \in {\mathbb R}$. 

First, we construct function $w_h$ in the form of (\ref{what}). 
Choose $\widetilde{\Og} = [-1.25, -1]\cup[1, 1.25]$.
The RBF center points $\widetilde{\xb}_l \in \widetilde{\Og}$ are set with uniform space $h = 1/32$. 
The shape parameter in (\ref{what}) is set as $\widetilde{\vep} = 1.4$.
To obtain $\widetilde{\lambda}_l$, we choose test points from the same set of center points $\widetilde{\xb}_l$ and impose the condition (\ref{w-BC}) to all test points. 
It shows that the RMS errors of approximating $w$  on $\widetilde{\Og}$ is $\|e_{w}\|_{\rm rms} = 2.4702e-10$. 
Then substituting $w_h(x)$ into (\ref{Lapw}) we can obtain the approximation of $(-\Dt)^\frac{\alpha}{2}w(x)$. 

Next, we move to solve the homogeneous problem \cref{Poisson-v}. 
In \cref{Tab4-2-3}, we present numerical errors in solution $u$ for different $\ap$ and $c^*$, where the RBF center points in \cref{ansatz} are uniformly distributed on $(-1, 1)$. 
\begin{table}[htbp]
\begin{center} 
\begin{tabular}{|c|c|c|c|c|c|c|}
\hline
 & \multicolumn{2}{|c|}{$\ap = 0.4$} & \multicolumn{2}{|c|}{$\ap = 1$} & \multicolumn{2}{|c|}{$\ap = 1.5$} \\
\cline{2-7}
$N$     & $c^* = 0.5$ & $c^* = 0.65$ & $c^* = 0.5$ & $c^* = 0.65$ & $c^* = 0.5$ & $c^* = 0.65$ \\
\hline
 7 & 1.809e-4 & 1.060e-6 & 5.472e-4 & 2.383e-5 & 1.092e-3 & 4.378e-5 \\
15 & 8.076e-8 & 4.043e-8 & 2.542e-7 & 1.027e-7 & 4.968e-7 & 2.083e-7 \\
31 &8.24e-10 &1.62e-10 & 2.997e-9 &3.04e-10 & 7.439e-9 & 6.12e-10 \\
\hline
\end{tabular}
\caption{RMS errors $\|e_u\|_{\rm rms}$ in solving problem (\ref{Poisson}) with exact solution $u(x) = 1/(1+x^2)$, where $h = 2/(N+1)$.}\label{Tab4-2-3}
\end{center}
\vspace{-.3cm}
\end{table}
Note that the condition number of the stiffness matrix only depends on the number of points $N$, power $\ap$, and constant $c^*$. 
Hence, the condition numbers for $c^* = 0.5$  are the same as those presented in \cref{Tab4-2-1}, while the condition numbers of $c^* = 0.65$ are much smaller.
\cref{Tab4-2-3} shows that the errors reduce quickly as the number of points $N$ increases. 
For $N = 31$, the error in constructing $w_h$ becomes dominant, and it stays around $10^{-10}$ even if $N$ is further increased. 
The effect of constant $c^*$ is similar to those observed \Cref{section-experiments} for problems with homogeneous boundary conditions.

\bigskip
\noindent{\bf Example 5.} Consider the 2D Poisson problem (\ref{Poisson}) on $\Og = (-1, 1)^2$, where the functions $f$ and $g$ are chosen as
\begin{eqnarray}
&&f(\xb) = 2^\ap\Gamma\big(1+\fl{\alpha}{2}\big)\Gamma\big(2+\frac{\ap}{2}\big) \,_2F_1\Big(2+\fl{\ap}{2}, 1+\fl{\ap}{2}; 2, -|\xb|^2\Big) x,\nonumber\\
&&g(\xb) = \frac{x}{1+|\xb|^2}.\nonumber
\end{eqnarray}
The exact solution is given by $u(\xb) = x/(1+|\xb|^2)$ for $\xb \in {\mathbb R}^2$. 

We first construct function $w_h(\xb)$. 
Choose $\widetilde{\Og} = [-{17}/{16},\, {17}/{16}]^2\backslash(-1,1)^2$, i.e. a layer outside of domain $\Og$ with a width of $1/16$. 
The RBF center points in (\ref{what}) are chosen as equally spaced grid points with $h = 1/32$, and 
the test points $\widetilde{\xb}'_k$ in (\ref{w-BC}) are taken from the same set of center points. 
\begin{table}[htbp]
\begin{center} 
\begin{tabular}{|c|c|c|c|c|c|c|c|}
\hline
 & & \multicolumn{2}{|c|}{$\ap = 0.4$} & \multicolumn{2}{|c|}{$\ap = 1$} & \multicolumn{2}{|c|}{$\ap = 1.5$} \\
\cline{3-8}
$N$    & $h$ & $c^* = 0.5$ & $c^* = 0.65$ & $c^* = 0.5$ & $c^* = 0.65$ & $c^* = 0.5$ & $c^* = 0.65$ \\
\hline
 $7^2$ & 1/4 & 4.875e-5 & 8.297e-5 & 7.986e-5 & 1.347e-5 & 9.413e-5 & 2.045e-4\\
$15^2$ &1/8 & 4.545e-6 & 4.567e-6 & 9.299e-6 & 6.723e-6 & 7.292e-6 & 1.151e-5\\
$31^2$ & 1/16 & 2.907e-6 & 1.840e-6 & 3.740e-6 & 2.230e-6 & 2.946e-6 & 1.711e-6\\
\hline
\end{tabular}
\caption{RMS errors $\|e_u\|_{\rm rms}$ in solving problem (\ref{Poisson}) with exact solution $u(\xb) = x/(1+|\xb|^2)$ for $\xb \in {\mathbb R}^2$.}\label{Tab4-2-4}
\end{center}
\vspace{-.3cm}
\end{table}
The shape parameter in \cref{what} is $\widetilde{\vep} = 1.4$. 
In this case, the RMS error in constructing $w_h$ is $\|e_{w}\|_{\rm rms} = 9.372e-8$. 
\Cref{Tab4-2-4} shows the RMS errors in  solution $u$ for different $\ap$ and $c^*$, where the RBF center points in \cref{ansatz} are uniformly distributed on $(-1, 1)^2$. 
Similar to 1D cases, the numerical error in solution $u$ is bounded by that in constructing function $w_h$. 
Generally, the error rates for nonhomogeneous cases depend on the accuracy of fitting boundary conditions along $\p\Og$.

\section{Conclusion}
\label{section-conclusion}
In this paper, we have introduced a new Gaussian radial basis functions (RBFs) collocation method for solving the fractional Poisson equation. 
Hypergeometric functions are employed to express the fractional Laplacian of Gaussian RBFs, thus leading to efficient assembly of matrices. 
Unlike the previously developed RBF-based methods, our method yields a stiffness matrix with a Toeplitz structure, facilitating the development of FFT-based fast algorithms for efficient matrix-vector multiplications.
A key focus of our study was the comprehensive investigation of shape parameter selection to address challenges related to ill-conditioning and numerical stability. 
In particular, our approach involves keeping $\gamma = (\vep h)^2$ constant, where $\vep$ represents the shape parameter, and $h$ is the spatial discretization parameter. 
Under this assumption, we provide rigorous error estimates for our method, filling a fundamental gap in the existing literature on RBF-based collocation methods.
Notably, to show numerical stability,  we establish a fractional Poincar\'e type inequality for globally supported Gaussian mixture functions.
Our analytical results reveal that the numerical error consists of a part that converges to zero at a rate dependent on the smoothness of functions, and a saturation error that depends on $\gamma$. 
We conduct numerical experiments to inform the practical selection of $\gamma$.

There are several future directions that need to be mentioned. 
Since the treatment for non-homogeneous Dirichlet boundary value problems depends on converting them to homogeneous ones through auxiliary functions, 
it is worthwhile to explore more efficient numerical algorithms and rigorous error analysis for constructing regular auxiliary functions, particularly in high dimensions. Future considerations may also involve exploring other types of boundary conditions. Our proposed method could potentially work for meshfree points, and it is of further interest to conduct comprehensive numerical studies in this respect.
Lastly, extending the method to solve other types of fractional PDEs is also worth further investigation.

\bigskip
\appendix
\section{Saturation errors}
\label{appendix}



We show the coefficients $a_{\bm\ap}^{(\bm\beta)}(\xb)/{\bm \ap}!$ of the saturation errors in Theorem \ref{them-conv} for different $\gamma$ and $x$ in one-dimensional cases. 
Tables \ref{Tab-A1} and \ref{Tab-A2} present the  coefficients for ${\bm \beta} = 0$ and ${\bm \beta} = 2$, respectively.
\begin{table}[htbp]
\begin{center} 
\begin{tabular}{|c|c|c|c|c|}
\hline
 & \multicolumn{2}{|c|}{$\gamma = 0.36$} & \multicolumn{2}{|c|}{$\gamma = 0.25$} \\
\cline{2-5}
   ${\bm\ap}$  & $x = 0.25$ & $x = 0.5$ & $x = 0.25$ & $x = 0.5$\\
   \hline
   0& 2.4808e-12 & 4.9617e-12& 1.4314e-17&  2.8629e-17 \\
   1& 2.1649e-11 & 0 & 1.7988e-16& 0  \\
   2& 9.4464e-11 & 1.8893e-10& 1.1302e-15& 2.2604e-15 \\
   3& 2.7478e-10 & 0 & 4.7342e-15 & 0 \\
   4& 5.9949e-10 &  1.1990e-09& 1.4873e-14& 2.9746e-14\\
   5& 1.0463e-09 & 0 &  3.7380e-14& 0  \\
   6& 1.5218e-09 & 3.0436e-09 & 7.8288e-14 & 1.5658e-13 \\
   7& 1.8971e-09 & 0 & 1.4054e-13 & 0 \\
   8& 2.0695e-09 & 4.1389e-09 & 2.2076e-13&   4.4153e-13\\
     \hline
\end{tabular}
\caption{Coefficients $|a_{\bm\ap}^{({\bm \beta})}(x)|/{\bm \ap}!$  of the saturation errors for ${\bm \beta}= 0$ in 1D  cases.}\label{Tab-A1}
\end{center}
\vspace{-.3cm}
\end{table}
Generally, for the same ${\bm \ap}$ and ${\bm \beta}$ the smaller the constant $\gamma$,  the smaller the coefficient $|a_{\bm\ap}^{({\bm \beta})}(x)|/{\bm \ap}!$

\begin{table}[htbp]
\begin{center} 
\begin{tabular}{|c|c|c|c|c|}
\hline
 & \multicolumn{2}{|c|}{$\gamma = 0.36$} & \multicolumn{2}{|c|}{$\gamma = 0.25$}\\ 
\cline{2-5}
   ${\bm \ap}$  & $x = 0.25$ & $x = 0.5$ & $x = 0.25$ & $x = 0.5$\\
   \hline
   0& 6.0278e-34 & 9.7940e-11 & 0 &  5.6511e-16\\
   1& 8.2351e-10 & 0 & 6.9215e-15 & 0 \\
   2& 2.4808e-12 & 3.4622e-09 & 1.4314e-17 &   4.2387e-14 \\
   3& 9.6826e-09 & 0 & 1.7288e-13 & 0 \\
   4& 9.4464e-11 & 2.0403e-08 & 1.1302e-15 & 5.2993e-13 \\
   5& 3.4048e-08 & 0 &  1.2935e-12 & 0  \\
   6& 5.9949e-10 & 4.8128e-08 & 1.4873e-14 & 2.6507e-12\\
   7& 5.6819e-08 & 0 & 4.6020e-12 & 0 \\
   8& 1.5218e-09 & 6.0903e-08 &  7.8288e-14 &  7.1059e-12\\
     \hline
\end{tabular}
\caption{Coefficients $|a_{\bm\ap}^{({\bm \beta})}(x)|/{\bm \ap}!$ of the saturation errors for ${\bm \beta}= 2$ in 1D cases.}\label{Tab-A2}
\end{center}
\vspace{-.3cm}
\end{table}

\Cref{FigA1} shows the coefficients of saturation errors versus $x$, where $\gamma = 0.36$. 
It shows that $|a_{\bm\alpha}^{(\bm\beta)}(x)|/{\bm\alpha}!$ is a periodic function of $x$.
\begin{figure}[htbp]
\centerline{
(a)\includegraphics[height = 6.56cm, width = 7.36cm]{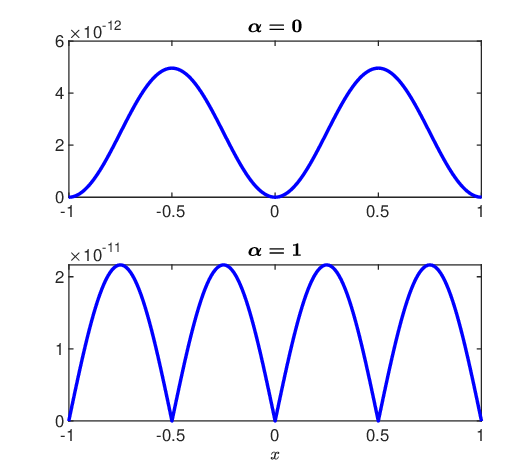}\hspace{-5mm}
(b)\includegraphics[height = 6.56cm, width = 7.36cm]{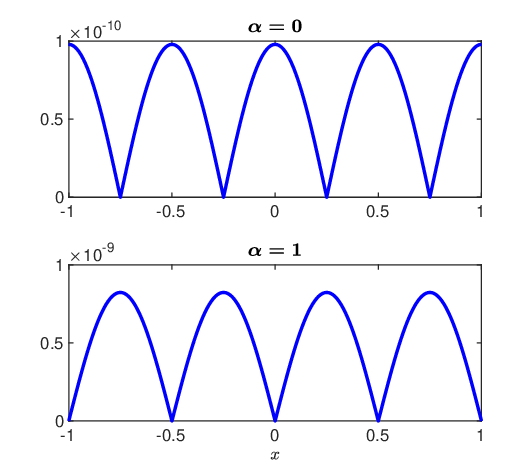}}
\caption{Coefficient of saturation errors (i.e. $|a_{\bm\alpha}^{(\bm\beta)}(x)|/{\bm\alpha}!$) in 1D cases. (a) ${\bm \beta} = 0$; (b)  ${\bm\beta} = 2$.}
\label{FigA1}
\vspace{-.3cm}
\end{figure}
 \Cref{FigA2} presents the relation of coefficient function $|a_{\bm\alpha}^{(\bm\beta)}(x)|/{\bm\alpha}!$ and $\gamma$, where $x = 0.25$ is fixed. 
\begin{figure}[htbp]
\centerline{
\includegraphics[height = 5.06cm, width = 6.86cm]{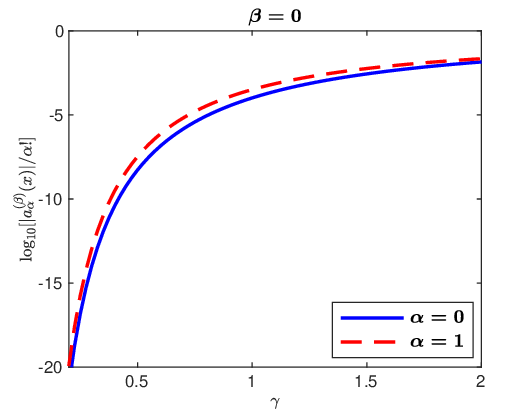}\hspace{-2mm}
\includegraphics[height = 5.06cm, width = 6.86cm]{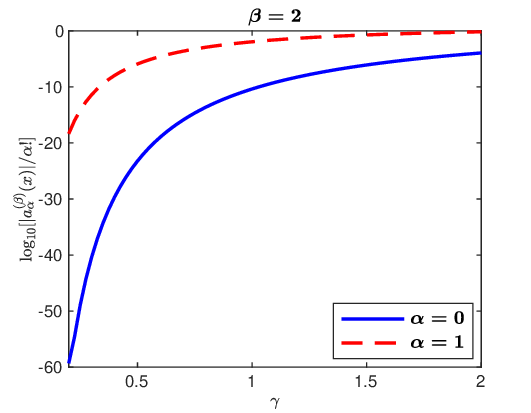}}
\caption{Coefficient of saturation errors  at $x = 0.25$ in 1D cases.}
\label{FigA2}
\vspace{-.3cm}
\end{figure}

\bibliographystyle{siamplain}
\bibliography{rbf, 2022-DiscreteFL}

\end{document}